\documentclass{siamart1116}
\usepackage{amsmath,amssymb}

\newcommand{\R}{{\mathbb R}}

\newcommand{\N}{{\mathbb N}}

\newcommand{\bL}{{\mathbb L}}

\newcommand{\cS}{{\mathcal S}}
\newcommand{\cL}{{\mathcal L}}

\newcommand{\cM}{{\mathcal{M}_{m_0,\kappa}}}

\newcommand{\Div}{\mathrm{div}}
\newcommand{\ph}{\varphi}

\newcommand{\1}{\chi}

\renewcommand{\geq}{\geqslant}
\renewcommand{\leq}{\leqslant}

\newtheorem{remark}[theorem]{Remark}

\numberwithin{theorem}{section}
\numberwithin{equation}{section}

\newcommand{\TheTitle}{Optimal location of resources for biased movement of species: the 1D case} 
\newcommand{\TheAuthors}{F. Caubet, T. Deheuvels, and Y. Privat}

\headers{Optimal location of resources for biased movement of species}{\TheAuthors}

\title{{\TheTitle}}

\author{Fabien Caubet\thanks{Institut de Math\'ematiques de Toulouse, Universit\'e de Toulouse, F-31062 Toulouse Cedex 9, France, ({\tt fabien.caubet@math.univ-toulouse.fr}).}
	\and  Thibaut Deheuvels\thanks{{\'E}cole normale sup{\'e}rieure de Rennes, Bruz, France, ({\tt thibaut.deheuvels@ens-rennes.fr}).}
        \and {Yannick Privat\thanks{CNRS, Universit\'e Pierre et Marie Curie (Univ. Paris 6), UMR 7598, Laboratoire Jacques-Louis Lions, F-75005, Paris, France ({\tt yannick.privat@upmc.fr}).}
 }}

\usepackage{amsopn}

\date{\today}

\begin{document}

\maketitle

\begin{abstract}
In this paper, we investigate an optimal design problem motivated by some issues arising in population dynamics. In a nutshell, we aim at determining the optimal shape of a region occupied by resources for maximizing the survival ability of a species in a given box and we consider the general case of Robin boundary conditions on its boundary. Mathematically, this issue can be modeled with the help of an extremal indefinite weight linear eigenvalue problem. The optimal spatial arrangement is obtained by minimizing the positive principal eigenvalue with respect to the weight, under a $\mathrm{L}^1$ constraint standing for limitation of the total amount of resources.
The specificity of such a problem rests upon the presence of nonlinear functions of the weight both in the numerator and denominator of the Rayleigh quotient. By using adapted rearrangement procedures, a well-chosen change of variable, as well as necessary optimality conditions, we completely solve this optimization problem in the unidimensional case by showing first that every minimizer is unimodal and {\it bang-bang}. This leads to investigate a finite dimensional optimization problem. This allows to show in particular that every minimizer is (up to additive constants) the characteristic function of three possible domains: an interval that sticks on the boundary of the box, an interval that is symmetrically located at the middle of the box, or, for a precise value of the Robin coefficient, all intervals of a given fixed length.
\end{abstract}

\begin{keywords}
principal eigenvalue, population dynamics, optimization, calculus of variations, rearrangement/symmetrization, bang-bang functions.
\end{keywords}

\begin{AMS}
49J15, 49K20, 34B09, 34L15.
\end{AMS}

\section{Introduction}

\subsection{The biological model} \label{Section:BiologicalModel}

In this paper, we consider a reaction-diffusion model for population dynamics. We assume that the environment is spatially heterogeneous, and present both favorable and unfavorable regions. More specifically, we assume that the intrinsic growth rate of the population is spatially dependent. Such models have been introduced in the pioneering work of Skellam \cite{MR0043440}, see also \cite{MR1112065,MR1105497} and references therein. 
We also assume that the population tends to move toward the favorable regions of the habitat, that is, we add to the model an advection term (or drift) along the gradient of the habitat quality. This model has been introduced by Belgacem and Cosner in \cite{MR1372792}.

More precisely, we assume that the flux of the population density $u(x,t)$ is of the form $-\nabla u + \alpha u \nabla m$, where $m(\cdot)$ represents the growth rate of the population, and will be assumed to be bounded and to change sign. From a biological point of view, the function $m(x)$ can be seen as a measure of the access to resources at a location $x$ of the habitat. The nonnegative constant $\alpha$ measures the rate at which the population moves up the gradient of the growth rate $m$. With a slight abuse of language, we will also say that $m( \cdot )$ stands for the local rate of resources or simply the resources at location $x$.

This leads to the following diffusive-logistic equation
\begin{equation}\label{RD0}
\left\{ \begin{array}{ll}
\partial_t u = \Div (\nabla u - \alpha u \nabla m) + \lambda u(m-u) &\text{in }~ \Omega\times (0,\infty),\\[1mm]
e^{\alpha m}(\partial_n u -\alpha u\partial_n m) +\beta u=0 &\text{on }~ \partial \Omega\times (0,\infty),
\end{array}
\right.
\end{equation}
where $\Omega$ is a bounded region of $\R^n$ ($n=1,2,3$) which represents the habitat, $\beta\geq0$, and $\lambda$ is a positive constant. The case $\beta=0$ in \eqref{RD0} corresponds to the no-flux boundary condition: the boundary acts as a barrier for the population. The Dirichlet case, where the boundary condition on $\partial \Omega$ is replaced by $u=0$, corresponds to the case when the boundary is lethal to the population, and can be seen as the limit case when $\beta\to \infty$. The choice $0<\beta<\infty$ corresponds to the case where a part of the population dies when reaching the boundary, while a part of the population turns back.

Plugging the change of function $v=e^{-\alpha m} u$ into Problem \eqref{RD0} yields to
\begin{equation}\label{RD}
\left\{ \begin{array}{ll}
\partial_t v = \Delta v + \alpha \nabla v \cdot \nabla m + \lambda v(m-e^{\alpha m}v) &\text{in }~ \Omega\times (0,\infty),\\[1mm]
e^{\alpha m}\partial_n v + \beta v=0 &\text{on }~ \partial \Omega\times (0,\infty).
\end{array}
\right.
\end{equation}
The relation $v=e^{-\alpha m}u$ ensures that the behavior of models \eqref{RD0} and \eqref{RD} in terms of growth, extinction or equilibrium is the same. Therefore, we will only deal with Problem \eqref{RD} in the following.

It would be natural {\it a priori} to consider weights $m$ belonging to $\mathrm{L}^\infty(\Omega)$ without assuming additional regularity assumption. Nevertheless, for technical reasons that will be made clear in the following, we will temporarily assume that $m\in C^2(\overline{\Omega})$. Moreover, we will also make the following additional assumptions on the weight $m$, motivated by biological reasons. Given $m_0\in (0,1)$ and~$\kappa>0$, we will consider that
\begin{itemize}
\item the {\bf total resources} in the heterogeneous environment are limited:
\begin{equation}\label{assumption1}
\int_\Omega m \leq -m_0|\Omega|,
\end{equation}
\item $m$ is a {\bf bounded measurable} function which {\bf changes sign} in $\Omega$, i.e. 
\begin{equation} \label{assumption3}
|\{x\in \Omega,~ m(x)>0\}| >0 ,
\end{equation}
and using an easy renormalization argument leads to assume that
\begin{equation}\label{assumption2}
-1\leq m \leq \kappa\quad \text{ a.e. in } \Omega .
\end{equation}
\end{itemize}
Observe that the combination of \eqref{assumption1} and \eqref{assumption3} guarantees that the weight $m$ changes sign in $\Omega$. 


In the following, we will introduce and investigate an optimization problem in which roughly speaking, one looks at configurations of resources maximizing the survival ability of the population. The main unknown will be the weight $m$ and for this reason, it is convenient to introduce the set of admissible weights
\begin{equation}\label{M}
\cM = \{m\in \mathrm{L}^\infty(\Omega),~ m \text{ satisfies assumptions } \eqref{assumption1},~ \eqref{assumption3} \text{ and } \eqref{assumption2}\}.
\end{equation}

\medskip

\paragraph{A principal eigenvalue problem with indefinite weight}
It is well known that the behavior of Problem \eqref{RD} can be predicted from the study of the following eigenvalue problem with indefinite weight (see \cite{MR1372792,MR1112065,MR1100011})
\begin{equation}\label{EVP0}
\left\{ \begin{array}{ll}
-\Delta \ph - \alpha \nabla m \cdot \nabla \ph = \Lambda m \ph & \text{in }~ \Omega,\\[1mm]
e^{\alpha m}\partial_n \ph +\beta \ph =0 &\text{on }~ \partial \Omega ,
\end{array}
\right.
\end{equation}
which also rewrites
\begin{equation}\label{EVPMultiD}
\left\{ \begin{array}{ll}
-\Div(e^{\alpha m}\nabla \ph) = \Lambda m e^{\alpha m} \ph & \text{in }~ \Omega,\\[1mm]
e^{\alpha m}\partial_n \ph +\beta \ph =0 &\text{on }~ \partial \Omega.
\end{array}
\right.
\end{equation}
Recall that an eigenvalue $\Lambda$ of Problem \eqref{EVPMultiD} is said to be a principal eigenvalue if $\Lambda$ has a positive eigenfunction. Using the same arguments as in \cite{MR1469392,MR588690}, the following proposition can be proved. For sake of completeness, we propose a sketch of the proof in Appendix \ref{proof-pp-ev}.

\begin{proposition}\label{pp-ev}
\begin{enumerate}
\item In the case of Dirichlet boundary condition, there exists a unique positive principal eigenvalue denoted $\lambda_1^\infty(m)$, which is characterized by
	\begin{equation}\label{def:lambda1beta-0}
	\lambda_1^\infty (m) = \inf_{\ph \in \cS_0} \frac {\int_\Omega e^{\alpha m} {|\nabla \ph|}^2} 
		{\int_\Omega m e^{\alpha m} \ph^2},
	\end{equation}
where $\cS_0 = \{\ph \in \mathrm{H}^1_0(\Omega),~ \int_\Omega m e^{\alpha m}\ph^2>0\}$.
\item In the case of Robin boundary condition with $\beta > 0$, the situation is similar to the Dirichlet case, and $\lambda_1^\beta(m)$ is characterized by
	\begin{equation}\label{def:lambda1beta}
	\lambda_1^\beta (m) = \inf_{\ph \in \cS} \frac {\int_\Omega e^{\alpha m} {|\nabla \ph|}^2 
		+ \beta \int_{\partial \Omega} \ph^2} {\int_\Omega m e^{\alpha m} \ph^2},
	\end{equation}
where $\cS = \{\ph \in \mathrm{H}^1(\Omega),~ \int_\Omega m e^{\alpha m}\ph^2>0\}$.
\item\label{Neumann} In the case of Neumann boundary condition ($\beta=0$),
\begin{itemize}
\item if $\int_\Omega m e^{\alpha m} < 0$, then the situation is similar as the Robin case, and $\lambda_1^\beta(m)>0$ is given by \eqref{def:lambda1beta} with $\beta=0$,
\item if $\int_\Omega m e^{\alpha m}\geq0$, then $ \lambda_1^\beta(m) =0$ is the only non-negative principal eigenvalue.
\end{itemize}
\end{enumerate}
\end{proposition}

Following \cite[Theorem 28.1]{MR588690} (applied in the special case where the operator coefficients are periodic with an arbitrary period), one has the following time asymptotic behavior characterization of the solution of the logistic equation \eqref{RD}: 
\begin{itemize}
\item if $\lambda>\lambda_1^\beta (m)$, then \eqref{RD} has a unique positive equilibrium, which is globally attracting among non-zero non-negative solutions,
\item if $\lambda_1^\beta (m) > 0$ and $0<\lambda<\lambda_1^\beta (m)$, then all non-negative solutions of \eqref{RD} converge to zero as $t\to \infty$.
 \end{itemize}

\begin{remark}
According to the existing literature (see e.g. \cite{MR1372792}), the existence of~$\lambda_1^\beta(m)$ defined as the principal eigenvalue of Problem \eqref{EVPMultiD} for $C^2$ weights follows from the Krein Rutman theory. Nevertheless, one can extend the definition of $\lambda^\beta_1(m)$ to a larger class of weights by using Rayleigh quotients, as done in Proposition~\ref{pp-ev} (see Remark \ref{rk:largelambda1}). 
\end{remark}

From a biological point of view, the above characterization yields a criterion for extinction or persistence of the species.

A consequence is that the smaller $\lambda_1^\beta (m)$ is, the more likely the population will survive. This biological consideration led Cantrell and Cosner to raise the question of finding $m$ such that $\lambda_1^\beta (m)$ is minimized, see \cite{MR1105497,MR1112065}. This problem writes
\begin{equation}\label{pb0907}
\inf_{m\in \cM}\lambda_1^\beta(m).
\end{equation}
or respectively
\begin{equation}\label{pb0907dir}
\inf_{m\in \cM}\lambda_1^\infty(m).
\end{equation}
in the case of Dirichlet conditions.

Biologically, this corresponds to finding the optimal arrangement of favorable and unfavorable regions in the habitat so the population can survive.

\begin{remark}
It is notable that, in the Neumann case ($\beta=0$), if we replace Assumption \eqref{assumption1} with $\int_\Omega m \geq 0$ in the definition of $\cM$, then $\lambda_1^0(m)=0$ for every~$m\in \cM$. Biologically, this means that any choice of distribution of the resources will ensure the survival of the population.
\end{remark}

\subsection{State of the art}
\paragraph{Analysis of the biological model (with an advection term)}
Problem \eqref{RD} was introduced in \cite{MR1372792}, and studied in particular in \cite{MR1372792, MR1961241}, where the question of the effect of adding the drift term is raised. The authors investigate if increasing $\alpha$, starting from $\alpha=0$, has a beneficial of harmful impact on the population, in the sense that it decreases or increases the principal eigenvalue of Problem \eqref{EVPMultiD}.

It turns out that the answer depends critically on the condition imposed on the boundary of the habitat.
Under Dirichlet boundary conditions, adding the advection term can be either favorable or detrimental to the population, see \cite{MR1372792}. This can be explained by the fact that if the favorable regions in the habitat are located near the hostile boundary, this could result in harming the population.
In contrast, under no-flux boundary conditions, it is proved in \cite{MR1372792} that a sufficiently fast movement up the gradient of the ressources is always beneficial. Also, according to \cite{MR1961241}, if we start with no drift ($\alpha=0$), adding the advection term is always beneficial if the habitat is convex. The authors however provide examples of non-convex habitats such that introducing advection up the gradient of $m$ is harmful to the population.

\medskip

\paragraph{Optimal design issues} The study of extremal eigenvalue problems with indefinite weights like Problem~\eqref{pb0907}, with slight variations on the parameter choices (typically~$\alpha=0$ or $\alpha>0$) and with different boundary conditions (in general Dirichlet, Neumann or Robin ones) is a long-standing question in calculus of variations. In the survey \cite[Chapter~9]{HenrotBook}, results of existence and qualitative properties of optimizers when dealing with non-negative weights are gathered. 

In the survey article \cite{Lou}, the biological motivations for investigating extremal problems for principal eigenvalue with sign-changing weights are recalled, as well as the first existence and analysis properties of such problems, mainly in the 1D case.

A wide literature has been devoted to Problem~\eqref{EVP0} (or close variants) without the drift term, \textit{i.e.} with $\alpha=0$. 
Monotonicity properties of eigenvalues and {\it bang-bang} properties of minimizers\footnote{It means that the $\mathrm{L}^\infty$ constraints on the unknown $m$ are saturated a.e. in $\Omega$, in other words that every optimizer~$m^*$ satisfies $m^*(x)\in \{-1,\kappa\}$ a.e. in $\Omega$.} were established in \cite{MR1469392}, \cite{MR2281509} and \cite{JhaPor11} for Neumann boundary conditions ($\beta=0$) in the 1D case. In~\cite{Roques-Hamel}, the same kind of results were obtained for periodic boundary conditions. We also mention~\cite{MR2660987}, for an extension of these results to principal eigenvalues associated to the one dimensional~$p$-Laplacian operator.

In this article, we will investigate a similar optimal design problem for a more general model in which a drift term with Robin boundary conditions is considered. In the simpler case where no advection term was included in the population dynamics equation, a fine study of the optimal design problem \cite{HinKaoLau12,LLNP2016} allowed to emphasize existence properties of {\it bang-bang} minimizers, as well as several geometrical properties they satisfy.
Concerning now the drift case model with Dirichlet or Neumann boundary conditions, the existence of principal eigenvalues and the characterization of survival ability of the population in terms of such eigenvalues has been performed in~\cite{MR1372792,MR1961241}.
However and up to our knowledge, nothing is known about the related optimal design problem~\eqref{pb0907} or any variant.

\medskip

\paragraph{Outline of the article}
This article is devoted to the complete analysis of Problem~\eqref{pb0907} in the 1D case, that is $\Omega=(0,1)$. In Section \ref{sec:model}, we discuss modeling issues and sum up the main results of this article. The precise (and then more technical) statements of these results are provided in Section \ref{sec:preciseResu} (Theorems \ref{thm:alpha}, \ref{thm:mainDN}, \ref{thm:main} and~\ref{thm:mainD}), as well as some numerical illustrations and consequences of these theorems. The whole section \ref{thm:alpha-proof} is devoted to proving Theorem \ref{thm:alpha} whereas the whole section \ref{sec:proofthm:mainD} is devoted to proving Theorems \ref{thm:mainDN}, \ref{thm:main} and~\ref{thm:mainD}. It is split into four steps that can be summed up as follows: (i) proof that one can restrict the search of minimizers to unimodal weights, (ii) proof of existence, (iii) proof of the {\it bang-bang} character of minimizers. The consequence of these three steps is that there exists a minimizer of the form $m^*=\kappa\1_E-\1_{\Omega\backslash E}$, where $E$ is an interval. The fourth step hence writes: (iv) optimal location of $E$ whenever $E$ is an interval of fixed length. 
Finally, we gather some conclusions and perspectives for ongoing works in Section \ref{sec:ccl}.

\subsection{Modeling of the optimal design problem and main results}\label{sec:model}

From now on, we focus on the 1D case $n=1$. Hence, for sake of simplicity, we will consider in the rest of the article that 
$$
\Omega = (0,1) .
$$
In the whole paper, if $\omega$ is a subset of $(0,1)$, we will denote by $\chi_\omega$ the characteristic function of $\omega$.

As mentioned previously (see Section~\ref{Section:BiologicalModel}), we aim at finding the optimal $m$ (whenever it exists) which minimizes the positive principal eigenvalue $\lambda_{1}^\beta (m)$ of Problem~\eqref{EVPMultiD}. 
For technical reasons, most of the results concerning the qualitative analysis of System \eqref{RD} (in particular, the persistence/survival ability of the population as $t\to +\infty$, the characterization of the principal eigenvalue $\lambda_1^\beta(m)$, and so on) are established by considering smooth weights, say~$C^2$. The following theorem emphasizes the link between the problem of minimizing $\lambda_1^\beta(m)$ over the class~$\cM\cap C^2(\overline{\Omega})$ and a relaxed one (as will be shown in the following), where one aims at minimizing $\lambda_1^\beta$ over the larger class $\cM$.

The following theorem will be made more precise in the following, and its proof is given at the end of Section \ref{subsec:bang-bang} below.

\begin{theorem}\label{thm:rennes1110}
When $\alpha$ is sufficiently small, the infimum $\inf \, \{ \,  \lambda_1^\beta(m) \, , \; m \,\in \,\cM\cap C^2(\overline{\Omega})\}$ is not attained for any $m\in \cM\cap C^2(\overline{\Omega})$. Moreover, one has
\begin{equation}\label{eq:rennes1128}
	\inf_{m\in \cM\cap C^2(\overline{\Omega})} \lambda_1^\beta(m) = \min_{m\in \cM} \lambda_1^\beta(m),
\end{equation}
and every minimizer $m^*$ of $\lambda_1^\beta$ over $\cM$ is a bang-bang function, \textit{i.e.} can be represented as $m^*=\kappa \1_E - \1_{\Omega\setminus E}$, where $E\subset \Omega$ is a measurable set.
\end{theorem}

As a consequence, throughout the paper, we consider the following optimization problem.

{\bf Optimal design problem. }\textit{
Fix $\beta\in [0,\infty]$.
We consider the extremal eigenvalue problem
\begin{equation}\label{defSOP}
\lambda_*^\beta=\inf \{\lambda_1^\beta (m), \ m\in \cM\},
\end{equation}
where $\cM$ is defined by \eqref{M} and where $\lambda_{1}^\beta (m)$ is the positive principal eigenvalue of
\begin{equation} \label{EVP}
\left\{ \begin{array}{ll}
-\left(e^{\alpha m}\ph' \right)'=\lambda me^{\alpha m}\ph & \text{in }~ (0,1),\\[1mm]
e^{\alpha m(0)}\ph'(0)=\beta \ph(0), \quad e^{\alpha m(1)}\ph'(1)=-\beta \ph(1). &
\end{array}
\right.
\end{equation}
Problem \eqref{EVP} above is understood in a weak sense, that is, in the sense of the variational formulation:
\begin{align}
&\text{Find $\varphi$ in $\mathrm{H}^1(0,1)$ such that for all $\psi \in \mathrm{H}^1(0,1)$} \label{EVPlambdastar},\\
&\int_0^1e^{\alpha m}\varphi' \psi'+\beta (\varphi(0)\psi(0)+\varphi(1)\psi(1))= \lambda_1^\beta(m) \int_0^1 m e^{\alpha m} \ph\psi. \notag
\end{align}}

%
%

\subsection{Solving of the optimal design problem \eqref{pb0907}}\label{sec:preciseResu}

Let us first provide a brief summary of the main results and the outline of this article.

\medskip

\paragraph{Brief summary of the main results}
In a nutshell, we prove that under an additional smallness assumption on the non-negative parameter $\alpha$, the problem of minimizing~$\lambda_1^\beta(\cdot)$ over $\cM$ has a solution writing
\begin{equation}\label{eq:m-bangbang}
m^*=\kappa\1_{E^*}-\1_{\Omega\backslash E^*},
\end{equation}
where $E^*$ is (up to a zero Lebesgue measure set) an interval. Moreover, one has the following alternative: except for one critical value of the parameter $\beta$ denoted $\beta_{\alpha,\delta}$, either $E^*$ is stuck to the boundary, or $E^*$ is centered at the middle point of $\Omega$. More precisely, there exists $\delta\in (0,1)$ such that:
\begin{list}{--}{\topsep1mm \itemsep1mm}
\item for {\bf Neumann boundary conditions}, one has $E^*=(0,\delta)$ or $E^*=(1-\delta,1)$;
\item for {\bf Dirichlet boundary conditions}, one has $E^*=((1-\delta)/2,(1+\delta)/2)$; 
\item for {\bf Robin boundary conditions}, there exists a threshold $\beta_{\alpha,\delta}>0$ such that, if $\beta<\beta_{\alpha,\delta}$ then the situation is similar to the Neumann case, whereas if $\beta>\beta_{\alpha,\delta}$ the situation is similar to the Dirichlet case.
\end{list}
Figure \ref{fig:illustrThm} illustrates different profiles of minimizers. The limit case $\beta=\beta_{\alpha,\delta}$ is a bit more intricate. For a more precise statement of these results, one refers to Theorems~\ref{thm:mainDN}, \ref{thm:main} and \ref{thm:mainD}.

\begin{figure}[h!]
\begin{center}
\includegraphics[width=4cm]{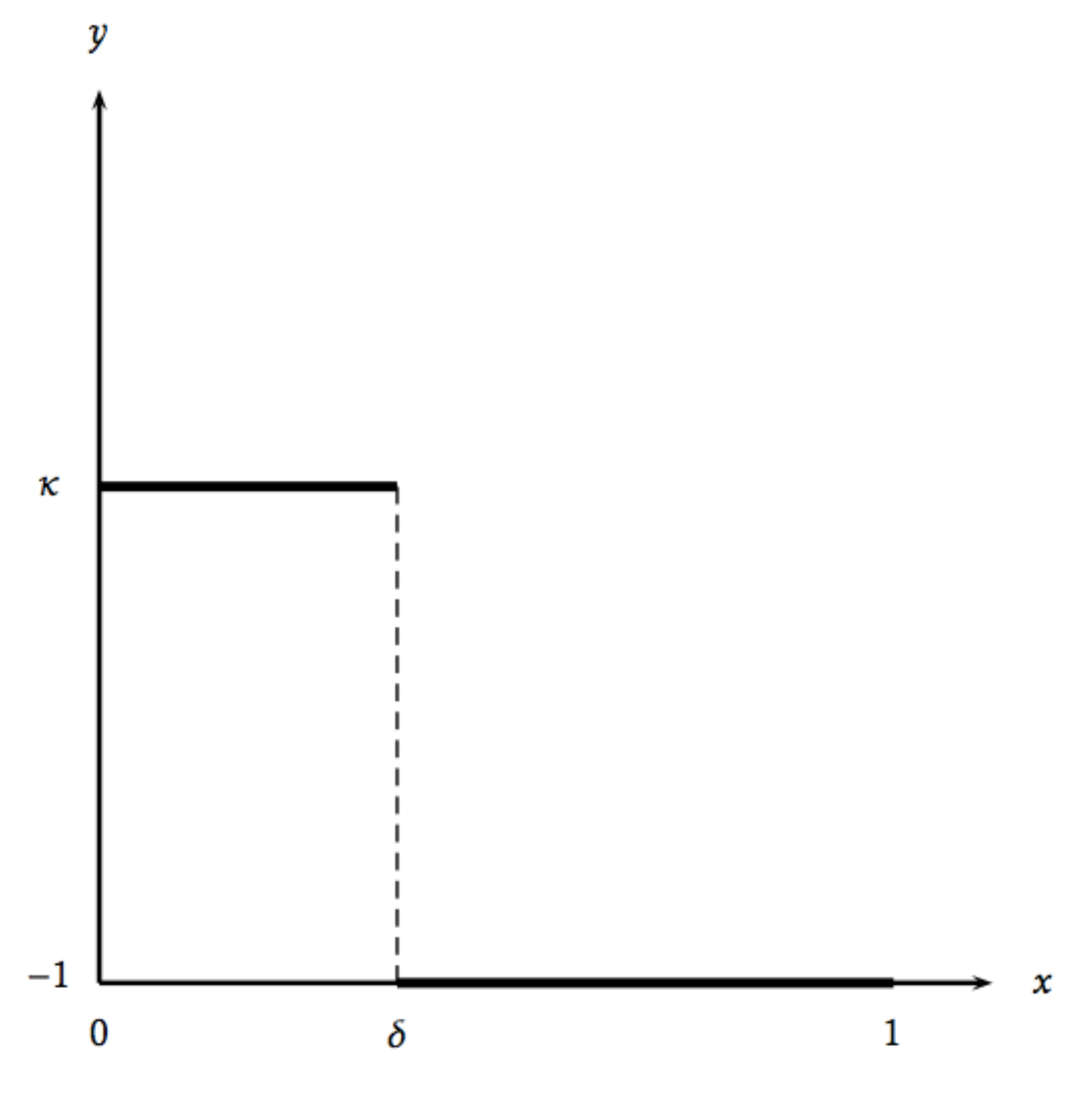}\hspace{1cm}
\includegraphics[width=4cm]{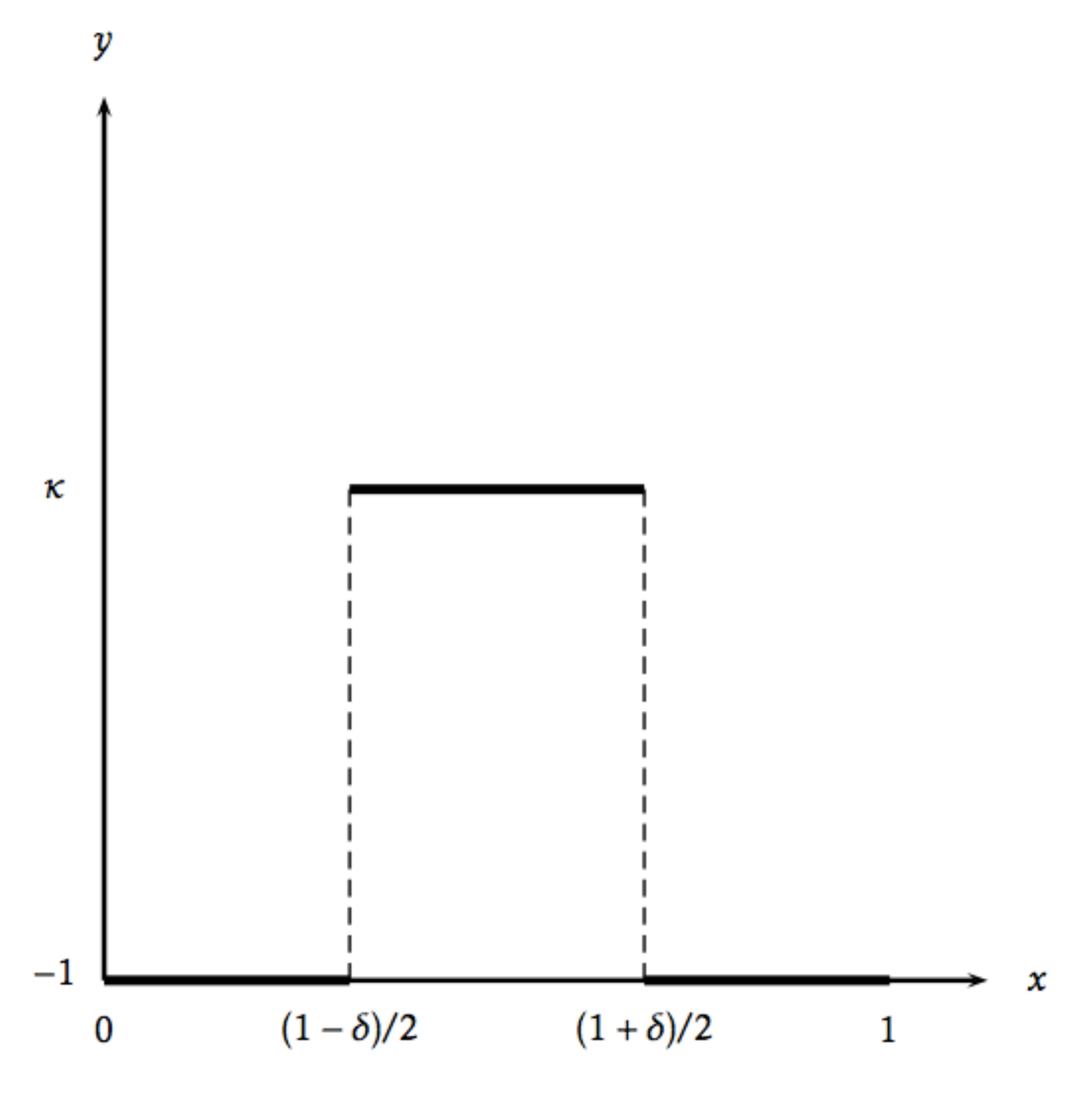}
\caption{Graph of a minimizer for small $\beta$ (left) and graph of the unique minimizer for large $\beta$  (right).\label{fig:illustrThm}}
\end{center}
\end{figure}

In this section, we will say that a solution $m_*^\beta$ (whenever it exists) of Problem~\eqref{defSOP} is of {\it Dirichlet type} if $m_*^\beta=(\kappa+1)\1_{((1-\delta)/2,(1+\delta)/2)}-1$ for some parameter~$\delta>0$.

We first investigate the Neumann and Robin cases. The Dirichlet case is a byproduct of our results on the Robin problem.

\medskip

\paragraph{Neumann boundary conditions}

In the limit case where Neumann boundary conditions are imposed (\textit{i.e.} $\beta=0$), one has the following characterization of persistence, resulting from the Neumann case in Proposition~\ref{pp-ev} (see \cite{MR1961241}).

\begin{proposition}\label{prop:alphastar}
Let $m\in \cM$. There exists a unique $\alpha^\star(m)>0$ such that 
\begin{list}{--}{\topsep1mm \itemsep0mm}
\item if $\alpha <\alpha^\star(m)$, then $\int_0^1 m e^{\alpha m} <0$ and $\lambda_1^0(m)>0$,
\item if $\alpha \geq\alpha^\star(m)$, then $\int_0^1 m e^{\alpha m} \geq 0$ and $\lambda_1^0(m)=0$.
\end{list}
\end{proposition}

As a consequence, in order to analyze the optimal design problem~\eqref{defSOP} which minimizes the positive principal eigenvalue $\lambda_1^\beta(m)$, it is relevant to consider (at least for the Neumann boundary conditions) $\alpha$ \textit{uniformly small with respect to} $m$. This is the purpose of the following theorem which is proved in Section~\ref{thm:alpha-proof} below.

\begin{theorem}[Neumann case]\label{thm:alpha}
The infimum
\begin{equation}\label{pbbaralpha}
\bar \alpha = \inf_{m\in \cM} \alpha^\star(m)
\end{equation}
is attained at every function $m_*\in \cM$ having the bang-bang property and such that $\int_\Omega m_*=-m_0$. In other words, the infimum is attained at every $m_*\in \cM$ which can be represented as $m_*=\kappa\1_E -\1_{\Omega\setminus E}$, where $E$ is a measurable subset of $\Omega$ of measure $(1-m_0)/(\kappa+1)$.
Moreover, one computes
$
\bar \alpha = \frac{1}{1+\kappa} \ln \left(\frac {\kappa+m_0}{\kappa(1-m_0)}\right) >0.
$
\end{theorem}

\begin{remark}\label{rk:paris1752}
A consequence of the combination of Theorem \ref{thm:alpha} and Proposition~\ref{pp-ev} is that $\int_\Omega me^{\alpha m}<0$ for every $m\in \cM$ whenever $\alpha <\bar{\alpha}$. 
\end{remark}

\begin{theorem}[Neumann case]\label{thm:mainDN}
Let $\beta=0$ 
and $\alpha\in [0,\bar\alpha)$. The optimal design problem~\eqref{defSOP} has a solution. 

If one assumes moreover that $\alpha\in [0,\min \{1/2,\bar\alpha\})$, then the inequality constraint \eqref{assumption1} is active, and the only solutions of Problem \eqref{defSOP} are $m=(\kappa+1)\1_{(0,\delta^*)}-1$ and $m=(\kappa+1)\1_{(1-\delta^*,1)}-1$,
where $\delta^* =  \frac{1-m_0}{\kappa +1} $.
\end{theorem}


\medskip

\paragraph{Robin boundary conditions}

The next result is devoted to the investigation of the Robin boundary conditions case, for an intermediate value of $\beta$ in $(0,+\infty)$. For that purpose, let us introduce the positive real number $\beta_{\alpha,\delta}$ such that
\begin{equation}\label{def:betaalphakappa}
\beta_{\alpha,\delta}=\left\{\begin{array}{ll}
\displaystyle \frac{e^{-\alpha}}{\sqrt{\kappa}\delta}\arctan \left(\frac{2\sqrt{\kappa}e^{\alpha(\kappa+1)}}{\kappa e^{2\alpha(\kappa+1)}-1}\right) & \quad \textrm{if }\kappa e^{2\alpha(\kappa+1)}>1 , \\
\displaystyle \frac{\pi e^{-\alpha}}{2\sqrt{\kappa}\delta} & \quad \textrm{if }\kappa e^{2\alpha(\kappa+1)}=1 , \\
\displaystyle \frac{e^{-\alpha}}{\sqrt{\kappa}\delta}\arctan \left(\frac{2\sqrt{\kappa}e^{\alpha(\kappa+1)}}{\kappa e^{2\alpha(\kappa+1)}-1}\right)+\frac{\pi e^{-\alpha}}{\sqrt{\kappa}\delta} & \quad \textrm{if }\kappa e^{2\alpha(\kappa+1)}<1.
\end{array}
\right.
\end{equation}
We also introduce 
$$\delta^*= \frac{1-m_0}{1+\kappa} ~~\text{ and } ~~ \xi^*=\frac{\kappa+m_0}{2(1+\kappa)},$$
and we denote by $\beta_\alpha^*$ the real number $\beta_{\alpha,\delta^*}$.

Note that the particular choice $|\{m=\kappa\}|=\delta^*$ corresponds to choosing $\int_0^1 m = -m_0$ if $m$ is {\it bang-bang}. It is also notable that if $E^*=(\xi^*,\xi^*+\delta^*)$ in \eqref{eq:m-bangbang}, then $\{m=\kappa\}$ is a centered subinterval of $(0,1)$.

\begin{theorem}[{Robin case}]\label{thm:main}
Let 
$\beta\geq 0$, and $\alpha\in [0,\bar\alpha)$. The optimal design problem~\eqref{defSOP} has a solution $m_*^\beta$. 

Defining $\delta =  \frac{1-\widetilde{m}_{0}}{\kappa +1} $, where $\widetilde{m}_{0}=-\int_0^1 m_*^\beta$ and assuming moreover that $\alpha\in [0,\min \{1/2,\bar\alpha\})$, one has the following.
\begin{itemize}
\item If $\beta <\beta_{\alpha,\delta}$, then $\int_0^1 m_*^\beta=-m_0$ and the solutions of Problem \eqref{defSOP} coincide with the solutions of Problem \eqref{defSOP} in the Neumann case.
\item If $\beta >\beta_{\alpha,\delta}$, then the solutions of Problem \eqref{defSOP} are of Dirichlet type. Moreover, if we further assume that 
\begin{equation}\label{eq:rennes1810}
	\alpha<\frac{\sinh^2{\big({\beta_{1/2}^*} \xi^*\big)}}{1+2\sinh^2{\big({\beta_{1/2}^*} \xi^*\big)}},
\end{equation}
then $\int_0^1 m_*^\beta=-m_0$ and the solutions of Problem \eqref{defSOP} coincide with the solutions of Problem \eqref{defSOP} in the Dirichlet case.
\item If $\beta =\beta_{\alpha,\delta}$, then $\int_0^1 m_*^\beta = -m_0$ and every function $m=(\kappa+1)\1_{(\xi,\xi+\delta^*)}-1$ where $\xi\in [0,1-\delta^*]$ solves Problem~\eqref{defSOP}.
\end{itemize}
\end{theorem}

This result is illustrated on Figure \ref{fig:illustrThm}. It can be seen as a generalization of \cite[Theorem 1]{LLNP2016}, where the case $\alpha=0$ is investigated. 

Let us comment on these results. It is notable that standard symmetrization argument cannot be directly applied. Indeed, this is due to the presence of the term~$e^{\alpha m}$ at the same time in the numerator and the denominator of the Rayleigh quotient defining $\lambda_1^\beta(m)$. 
The proofs rest upon the use of a change of variable to show some monotonicity properties of the minimizers, combined with an adapted rearrangement procedure as well as a refined study of the necessary first and second order optimality conditions to show the {\it bang-bang} property of the minimizers.

Let us now comment on the activeness of the inequality constraint \eqref{assumption1}. In the case $\alpha=0$, one can prove that a comparison principle holds (see \cite{MR2281509}, Lemma 2.3). A direct consequence is that the constraint \eqref{assumption1} is always active. In our case however, it can be established that the comparison principle fails to hold, and the activeness of the constraint has to be studied {\it a posteriori}.

\begin{remark}
Note that under the assumptions of Theorem \ref{thm:main}, with the additional assumption \eqref{eq:rennes1810}, Theorem \ref{thm:main} rewrites:
\begin{list}{--}{\topsep1mm \itemsep0mm}
\item if $\beta<\beta_\alpha^*$, then the only solutions of Problem \eqref{defSOP} are the Neumann solutions;
\item if $\beta>\beta_\alpha^*$, then the only solution of Problem \eqref{defSOP} is the Dirichlet solution;
\item if $\beta=\beta_\alpha^*$, then every function $m=(\kappa+1)\1_{(\xi,\xi+\delta^*)}-1$ where $\xi\in [0,1-\delta^*]$ solves Problem~\eqref{defSOP}.
\end{list}
\end{remark}
\begin{remark}
We can prove that, if assumption \eqref{eq:rennes1810} fails to hold, then there exist sets of parameters such that $\int_0^1m_*^\beta < m_0$.
\end{remark}

\medskip

\paragraph{Dirichlet boundary conditions}

Finally, as a byproduct of Theorem \ref{thm:main}, we have the following result in the case of Dirichlet boundary conditions. 

\begin{theorem}[Dirichlet case]\label{thm:mainD}
Let $\beta=+\infty$ 
and $\alpha \geq 0$. The optimal design problem \eqref{defSOP} has a solution. 
If one assumes moreover that $\alpha\in [0, 1/2)$, then any solution of Problem \eqref{defSOP} writes $m=(\kappa+1)\1_{((1-\delta)/2,(1+\delta)/2)}-1$ for some $\delta\in (0,1)$.
\end{theorem}

\subsection{Qualitative properties and comments on the results}

It is interesting to notice that, according to the analysis performed in Section \ref{sec:appendoptlocint} (see~\eqref{Fet0825} and \eqref{et0822}) the optimal eigenvalue $\lambda^\beta_{*}$ is the first positive solution of an algebraic equation, the so-called {\it transcendental equation}. More precisely, 
\begin{itemize}
\item {\it in the case $\beta <\beta_{\alpha,\delta}$,} the optimal eigenvalue $\lambda_*^\beta$ is the first positive root of the equation (of unknown $\lambda$)
$$\textstyle
\tan \big(\sqrt{\lambda \kappa}\delta \big)=\sqrt{\kappa}e^{\alpha (\kappa+1)}\frac{(\lambda+\beta^2e^{2\alpha})\tanh\left(\sqrt{\lambda}(1-\delta)\right)+2\beta e^\alpha\sqrt{\lambda}}{\beta e^\alpha \sqrt{\lambda}(\kappa e^{2 \alpha (\kappa+1)}-1)\tanh\left(\sqrt{\lambda}(1-\delta)\right)+e^{2\alpha}(\lambda \kappa e^{2\alpha \kappa}-\beta^2)},
$$ 
\item {\it in the case $\beta >\beta_{\alpha,\delta}$,} the optimal eigenvalue $\lambda_*^\beta$ is the first positive root of the equation (of unknown $\lambda$)
$$\textstyle
\tan \big(\sqrt{\lambda \kappa}\delta \big)=\sqrt{\kappa}e^{\alpha (\kappa+1)}\frac{(\lambda+\beta^2e^{2\alpha})\sinh(\sqrt{\lambda}(1-\delta))+2\beta\sqrt{\lambda}e^\alpha\cosh(\sqrt{\lambda}(1-\delta))}{\mathcal{D}_{\alpha}(\beta,\lambda)} ,
$$ 
where
\begin{multline*}
 \mathcal{D}_{\alpha}(\beta,\lambda) = \frac{1}{2}(\kappa e^{2\alpha (1+\kappa)}-1)(\beta^2e^{2\alpha}+\lambda)\cosh (\sqrt{\lambda}(1-\delta))\\
~~~+ \beta e^\alpha\sqrt{\lambda}(\kappa e^{2 \alpha(\kappa+1)}-1) \sinh (\sqrt{\lambda}(1-\delta))+\frac{1}{2}(1+\kappa e^{2\alpha (1+\kappa)})(\lambda-\beta^2e^{2\alpha}).
\end{multline*}
\end{itemize}
These formulae provide an efficient way to compute the numbers $\lambda_*^\beta$ since it comes to the resolution of a one-dimensional algebraic equation. 

On Figure \ref{fig:graphBetaLambdastar}, we used this technique to draw the graph of $\beta \mapsto \lambda_*^\beta$ for a given choice of the parameters $\alpha$, $\kappa$ and $m_0$.
From a practical point of view, we used a Gauss-Newton method on a standard desktop machine.

It is notable that one can recover from this figure, the values $\lambda_*^0$ (optimal value of $\lambda_1$ in the Neumann case) as the ordinate of the most left hand point of the curve and $\lambda_*^\infty$ (optimal value of~$\lambda_1$ in the Dirichlet case) as the ordinate of all points of the horizontal asymptotic axis of the curve. 

Finally, the concavity of the function $\beta \mapsto \lambda_*^\beta$ can be observed on Figure \ref{fig:graphBetaLambdastar}. This can be seen as a consequence of the fact that $\lambda_*^\beta$ writes as the infimum of linear functions of the real variable $\beta$.  

\begin{figure}[H]
\begin{center}
\includegraphics[width=5cm]{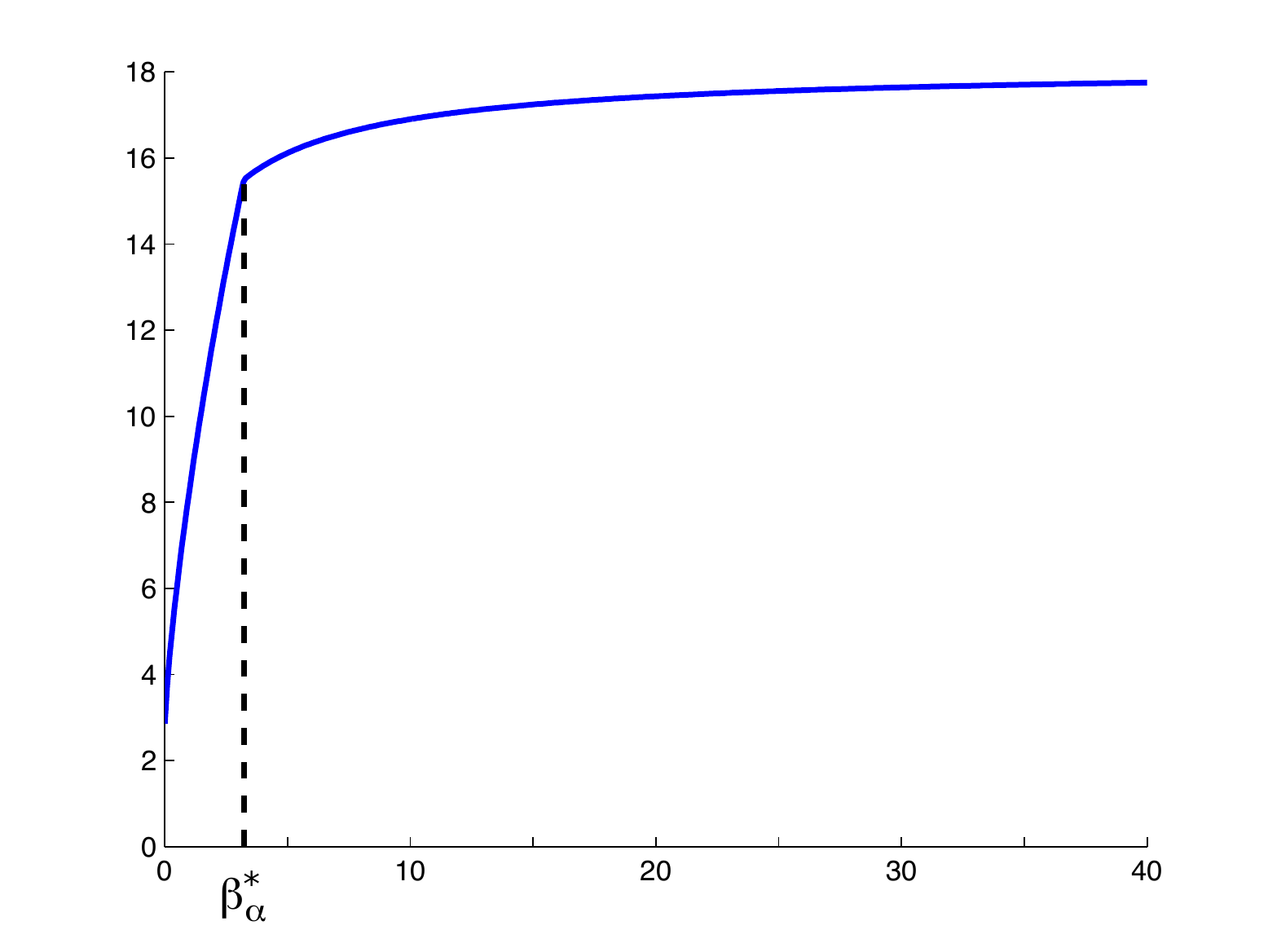}
\caption{Graph of $\beta\mapsto \lambda_*^\beta$ for $\alpha=0.2$, $\kappa=1$ and $m_0=0.4$. In that case, $\beta_\alpha^*\simeq 3.2232$.\label{fig:graphBetaLambdastar}}
\end{center}
\end{figure} 

%

\section{Proof of Theorem \ref{thm:alpha}}\label{thm:alpha-proof}

In view of Proposition \ref{pp-ev}(\ref{Neumann}), we start by maximizing $\int_0^1 me^{\alpha m}$ over $\cM$.

\begin{lemma}\label{lem:pb-max}
The supremum
\begin{equation}\label{pb-max}
\sup_{m\in \cM} \int_0^1 me^{\alpha m}
\end{equation}
is attained at some $m\in \cM$. Moreover, if $m$ is a maximizer of \eqref{pb-max}, then $m$ is bang-bang, \textit{i.e.} can be represented as $m=\kappa \1_E - \1_{(0,1)\setminus E}$, where $E$ is a measurable set in $(0,1)$, and $\int_0^1 m = -m_0$.
\end{lemma}
 
\begin{proof}
We first consider a problem similar to \eqref{pb-max}, where we remove the assumption that $m$ should change sign in $(0,1)$, namely we consider the maximization problem
\begin{equation}\label{pb-max-2} 
	\sup_{m\in \widetilde{\mathcal{M}}_{m_0,\kappa}} \int_0^1 me^{\alpha m}
\end{equation}
where $\widetilde{\mathcal{M}}_{m_0,\kappa} = \{m\in \mathrm{L}^\infty(0,1),~ m \text{ satisfies assumptions } \eqref{assumption1} \text{ and } \eqref{assumption2}\}$.

\paragraph{Step 1. Restriction to monotone functions} We claim that the research of a maximizer for Problem \eqref{pb-max-2} can be restricted to the monotone non-increasing functions of~$\widetilde{\mathcal{M}}_{m_0,\kappa}$. Indeed, if $m\in \widetilde{\mathcal{M}}_{m_0,\kappa}$, we introduce its monotone non-increasing rearrangement $m^\searrow$ (see e.g.~\cite{MR2455723} for details). By the equimeasurability property of monotone rearrangements, one has $\int_0^1 m^\searrow = \int_0^1 m$. Since it is obvious that $m^{\searrow}$ also satisfies Assumption \eqref{assumption2}, one has $m^\searrow\in \widetilde{\mathcal{M}}_{m_0,\kappa}$. Moreover, the equimeasurability property also implies that $\int_0^1 m^\searrow e^{\alpha m^\searrow} = \int_0^1 me^{\alpha m}$, which concludes the proof of the claim.

\paragraph{Step 2. Existence of solutions} Let us now show that there exists a maximizer for Problem \eqref{pb-max-2}. To see this, we consider a maximizing sequence $m_k$ associated with Problem \eqref{pb-max-2}. By the previous point, we may assume that the functions $m_k$ are non-increasing. Helly's selection theorem ensures that, up to a subsequence,~$m_k$ converges pointwise to a function $m^*$. Hence, $-1\leq m^* \leq \kappa$ a.e. in $(0,1)$, and $\color{red}\int_0^1 m^* \leq -m_0$ by the dominated convergence theorem, which implies that $m^*\in \widetilde{\mathcal{M}}_{m_0,\kappa}$.
Using the dominated convergence theorem again, we obtain that $\int_0^1 m_k e^{\alpha m_k} \to \int_0^1 m^* e^{\alpha m^*}$ as~$k\to \infty$. Therefore, $m^*$ is a maximizer of \eqref{pb-max-2}.

\paragraph{Step 3. Optimality conditions and bang-bang properties of maximizers} We now prove that every maximizer $m^*$ of Problem \eqref{pb-max-2} is {\it bang-bang}. Note that since $m^*$ is {\it bang-bang} if and only if its monotone non-increasing rearrangement is {\it bang-bang}, we may assume that $m^*$ is non-increasing. As a consequence, we aim at proving that~$m^*$ can be represented as $m^*=(\kappa +1) \1_{(0,\gamma)}-1$ for some $\gamma\in (0,1)$.
\\
We assume by contradiction that $|\{-1<m^*<\kappa\}|>0$. We will reach a contradiction using the first order optimality conditions. Introduce the Lagrangian function~$\cL$ associated to Problem \eqref{pb-max}, defined by
$$
	\cL : (m,\mu)\in \widetilde{\mathcal{M}}_{m_0,\kappa}\times \R \mapsto \int_0^1 m e^{\alpha m}  -\eta\left(\int_0^1 m(x)\; dx +m_0 \right).
$$
Denote by $\eta^*$ the Lagrange multiplier associated to the constraint $\int_0^1 m \leq -m_0$. Since we are dealing with an inequality constraint, we have $\eta^*\geq 0$. If $x_0$ lies in the interior of the interval $\{-1<m^*<\kappa\}$ and $h = \1_{(x_0-r,x_0+r)}$, then we observe that~$m^*+rh\in \widetilde{\mathcal{M}}_{m_0,\kappa}$ and $m^*-rh\in \widetilde{\mathcal{M}}_{m_0,\kappa}$ if $r>0$ is small enough. The first order optimality conditions then yield that $\langle d_m\mathcal{L}(m^*,\mu^*),h\rangle = 0$, that is
$$
	\int_0^1 h\big(e^{\alpha m^*}(1+\alpha m^*)-\eta^*\big) = 0.
$$
Consequently, the Lebesgue Density Theorem ensures that $e^{\alpha m^*}(1+\alpha m^*)=\eta^*$ a.e. in $\{-1<m^*<\kappa\}$. Studying the function $y\mapsto e^{\alpha y}(1+\alpha y)$ yields that $m^*$ is equal to a constant $\zeta\in [-1/\alpha,\kappa)$ in $\{-1<m^*<\kappa\}$. Therefore, $m^*$ can be represented as~$m^*=\kappa \1_{[0,\gamma_1]} + \zeta \1_{(\gamma_1,\gamma_2)} - \1_{[\gamma_2,1]}$, where $0\le\gamma_1 < \gamma_2\le 1$. Let us show that one has necessarily $\gamma_1=\gamma_2$, by constructing an admissible perturbation which increases the cost function whenever $\gamma_1<\gamma_2$. For $\theta>0$, we introduce the function $m^*_\theta$ defined by
$$
m^*_\theta = \kappa \1_{[0,\gamma_1^\theta]} + \zeta \1_{(\gamma_1^\theta,\gamma_2^\theta)} - \1_{[\gamma_2^\theta,1]},
$$
where $\gamma_1^\theta = \gamma_1+(1+\zeta)\theta$ and $\gamma_2^\theta = \gamma_2-(\kappa-\zeta)\theta$. Note that $\int_0^1 m^*_\theta = \int_0^1 m^*$ and~$m^*_\theta\in [-1,\kappa]$ a.e. in $(0,1)$, which implies that $m^*_\theta\in \widetilde{\mathcal{M}}_{m_0,\kappa}$ if $\theta$ is sufficiently small. One computes
$$
	\int_0^1 (m^*_\theta e^{\alpha m^*_\theta} - m^*e^{\alpha m^*}) ~=~  
	\theta \left( (1+\zeta)(\kappa e^{\alpha \kappa} - \zeta e^{\alpha \zeta}) 
	- (\kappa-\zeta)(e^{-\alpha}+\zeta e^{\alpha \zeta}) \right).
$$
Setting $\psi : \zeta \mapsto (1+\zeta)(\kappa e^{\alpha \kappa} - \zeta e^{\alpha \zeta}) - (\kappa-\zeta)(e^{-\alpha}+\zeta e^{\alpha \zeta})$, one has $\psi''(\zeta)=-\alpha e^{\alpha\zeta} (1+\kappa)(2+\alpha\zeta)$, from which we deduce that $\psi$ is strictly concave in $[-1/\alpha,\kappa]$. Since~$\psi'(-1/\alpha)=0$, $\psi'(\kappa)<0$ and $\psi(\kappa)=0$, we obtain that $\psi(\zeta)>0$ for all $\zeta\in [-1/\alpha,\kappa)$. As a consequence, if $\theta$ is small enough, then $m^*_\theta \in \widetilde{\mathcal{M}}_{m_0,\kappa}$ and $\int_0^1 m^*_\theta e^{\alpha m^*_\theta}> \int_0^1 m^* e^{\alpha m^*}$, which is a contradiction.

We have then proved that $m^*$ writes $m^*=(\kappa +1) \1_E -1$ for some measurable set~$E\subset (0,1)$. As a consequence, $\int_0^1 m^* e^{\alpha m^*}$ is maximal when $|E|$ is maximal, that is, when $|E| = (1-m_0)/(\kappa+1)$, which corresponds to $\int_0^1 m^* = -m_0$.
Since $\int_0^1 m^* e^{\alpha m^*}$ does not depend on the set $E$ in the representation $m^*=(\kappa+1)\1_E -1$, we deduce that every {\it bang-bang} function in $\widetilde{\mathcal{M}}_{m_0,\kappa}$ satisfying $\int_0^1 m^*=-m_0$ is a maximizer of Problem \eqref{pb-max-2}.

To conclude, observe that because of Assumption \eqref{assumption1}, every {\it bang-bang} function~$m^*$ in $\widetilde{\mathcal{M}}_{m_0,\kappa}$ satisfying $\int_0^1 m^*=-m_0$ changes sign, which implies that one has in fact~$m^*\in \cM$. This concludes the proof.
\end{proof}

We can now prove Theorem \ref{thm:alpha}. Denote by $m$ any {\it bang-bang} function of $\cM$ satisfying $\int_0^1 m=-m_0$ and by $m^*$ their decreasing rearrangement. According to the first step of the proof of Lemma \ref{lem:pb-max}, there holds $\alpha^\star(m)=\alpha^\star(m^*)$, where $\alpha^\star$ is defined in Proposition~\ref{prop:alphastar}. Moreover, using that $m^*=(\kappa+1)\1_{(0,\gamma)} -1$ with $\gamma=(1-m_0)/(\kappa+1)$, a quick computation shows that~$\alpha^\star(m^*) = \frac{1}{1+\kappa} \ln \frac {\kappa+m_0}{\kappa(1-m_0)}$. Indeed,  $\alpha^\star(m^*) $ is reached.
Consider $\alpha<\alpha^\star(m^*)$, which implies that $\int_0^1 m^* e^{\alpha m^*} <0$. We can apply Lemma~\ref{lem:pb-max} which yields that $\int_0^1 m e^{\alpha m} <0$ for every $m\in \cM$. We then deduce that $\alpha<\alpha^\star(m)$ for every $m\in \cM$. Therefore, one has $\alpha^\star(m^*) \leq \alpha^\star(m)$ for all $m\in \cM$, which proves that $\bar \alpha = \alpha^\star(m^*)$ and concludes the proof of Theorem~\ref{thm:alpha}.

\section{Proofs of Theorems \ref{thm:mainDN}, \ref{thm:main} and~\ref{thm:mainD}}\label{sec:proofthm:mainD}

Since the proof of Theorem \ref{thm:main} can be considered as a generalization of the proofs of Theorems~\ref{thm:mainDN} and~\ref{thm:mainD}, we will only deal with the general case of Robin boundary conditions (\textit{i.e.} $\beta \in [0,+\infty]$) in the following. The proofs in the Neumann and Dirichlet cases become simpler since the rearrangement to be used is standard (monotone rearrangement in the Neumann case and Schwarz symmetrization in the Dirichlet case). This is why in such cases, the main simplifications occur in Section \ref{subsec:unimodal} where one shows that a minimizer function~$m_*^\beta$ is necessary {\it unimodal} (in other words, $m_*^\beta$ is successively non-decreasing and then non-increasing on $(0,1)$).   


\subsection{Every minimizer is unimodal}\label{subsec:unimodal}
We will show that the research of minimizers can be restricted to unimodal functions of $\cM$.
 
Take a function $m\in \cM$. We will construct a unimodal function $m^R \in \cM$ such that $\lambda_1^\beta(m^R)\leq \lambda_1^\beta(m)$, where the inequality is strict if $m$ is not unimodal.
\\
We denote by $\ph$ the eigenfunction associated to $m$, in other words the principal eigenfunction solution of Problem \eqref{EVPlambdastar}.
According to the Courant-Fischer principle, there holds
\begin{equation}\label{lambCF}
\lambda_1^\beta(m)=\Re_{m}^\beta[\varphi]=\min_{\substack{\ph\in \mathrm{H}^1(0,1)\\ \int_0^1 m e^{\alpha m}\ph^2 >0}}\Re_{m}^\beta[\varphi] ,
\end{equation}
where 
\begin{equation}\label{def:RRayleigh}
\Re_{m}^\beta[\varphi]=\frac{\int_0^1 e^{\alpha m(x)}\ph'(x) ^2\, dx+\beta \ph(0)^2+\beta \ph (1)^2}{\int_0^1 m(x) e^{\alpha m(x)}\ph(x)^2\, dx}.
\end{equation}

\subsubsection{A change of variable}\label{sec:cdv}
Let us consider the change of variable 
\begin{equation}\label{cdv}
y=\int_0^x e^{-\alpha m(s)}\, ds, \quad x\in [0,1].
\end{equation}
The use of such a change of variable is standard when studying properties of the solutions of Sturm-Liouville problems (see e.g. \cite{MR1423004}).

Noting that $y$ seen as a function of $x$ is monotone increasing on $[0,1]$, let us introduce the functions $c$, $u$ and $\tilde{m}$ defined by
\begin{equation}\label{cofm}
c(x)=\int_0^x e^{-\alpha m (s)}\, ds, \quad  u(y)=\ph(x),\quad \text{and}\quad \tilde{m}(y)=m(x),
\end{equation}
for $x\in [0,1]$ and $y\in [0,c(1)]$.

Notice that
$
\int_0^{c(1)}\tilde{m}(y)e^{\alpha \tilde{m} (y)}\, dy=\int_0^1m (x)\, dx\leq -m_0.
$
Let us introduce 
\begin{equation}\label{def:x+}
x^+=\min \underset{x\in [0,1]}{\operatorname{argmax}}\ \varphi (x),
\end{equation}
in other words $x^+$ denotes the first point of $[0,1]$ at which the function $\varphi$ reaches its maximal value. 
We will also need the point $y^+$ as the range of $x^+$ by the previous change of variable, namely 
\begin{equation}\label{def:c+}
y^+=\int_0^{x^+} e^{-\alpha m(s)}\, ds.
\end{equation}

\subsubsection{Rearrangement inequalities} 
 
Using the change of variable \eqref{cdv} allows to write
$$
\lambda_1^\beta(m) = \Re_{m}^\beta [\varphi]=\frac{N_1+N_2}{D_1+D_2}\\
$$
with 
$$
\begin{array}{ll}
\displaystyle N_1 = \int_0^{y^+}u'(y)^2\, dy+\beta u(0)^2, & \quad \displaystyle 
N_2 = \int_{y^+}^{c(1)}u'(y)^2\, dy+\beta u(c(1))^2,\\
\displaystyle  D_1 = \int_0^{y^+}\tilde{m}(y)e^{2\alpha\tilde{m}(y)}u(y)^2\, dy, & \quad  \displaystyle 
D_2 = \int_{y^+}^{c(1)}\tilde{m}(y)e^{2\alpha\tilde{m}(y)}u(y)^2\, dy.
\end{array}
$$
\paragraph{Step 1. Unimodal rearrangements} Introduce the function 
$u^R$ defined on $(0,c(1))$ by
$$
u^R(y)=\left\{\begin{array}{ll}
u^\nearrow(y) & \textrm{on }(0,y^+),\\
u^\searrow(y) & \textrm{on }(y^+,c(1)),
\end{array}
\right.
$$
where $u^{\nearrow}$ denotes the monotone increasing rearrangement\footnote{Recall that, for a given function $v\in \bL(0,L)$ with $L>0$, one defines its monotone increasing rearrangement $v^{\nearrow}$ for a.e. $x\in (0,L)$ by $v^{\nearrow}(x)=\sup \{c\in \R \mid x\in \Omega_c^*\}$, where $\Omega_c^*=(1-|\Omega_c|,1)$ with $\Omega_c=\{v>c\}$.} of $u$ on $(0,y^+)$ and 
$u^\searrow$ denotes the monotone decreasing rearrangement\footnote{Similarly, $v^\searrow$ is defined by $v^\searrow(x)=v^{\nearrow}(1-x)$.} of $u$ on $(y^+,c(1))$ (see for instance~\cite{MR810619,MR2455723} for details and see Figure~\ref{illustrationSym} for an illustration of this procedure). Thanks to the choice of~$y^+$, it is clear that this rearrangement does not introduce discontinuities, and more precisely that $u^R\in \mathrm{H}^1(0,c(1))$.\\
Similarly, we also introduce the rearranged weight $\tilde{m}^R$, defined by 
$$
 \tilde{m}^R(y)=\left\{\begin{array}{ll}
\tilde m^\nearrow(y) & \textrm{on }(0,y^+),\\
\tilde m^\searrow(y) & \textrm{on }(y^+,c(1)),
\end{array}
\right.
$$
with the same notations as previously. 
\begin{figure}[H]
\begin{center}
\includegraphics[width=5cm]{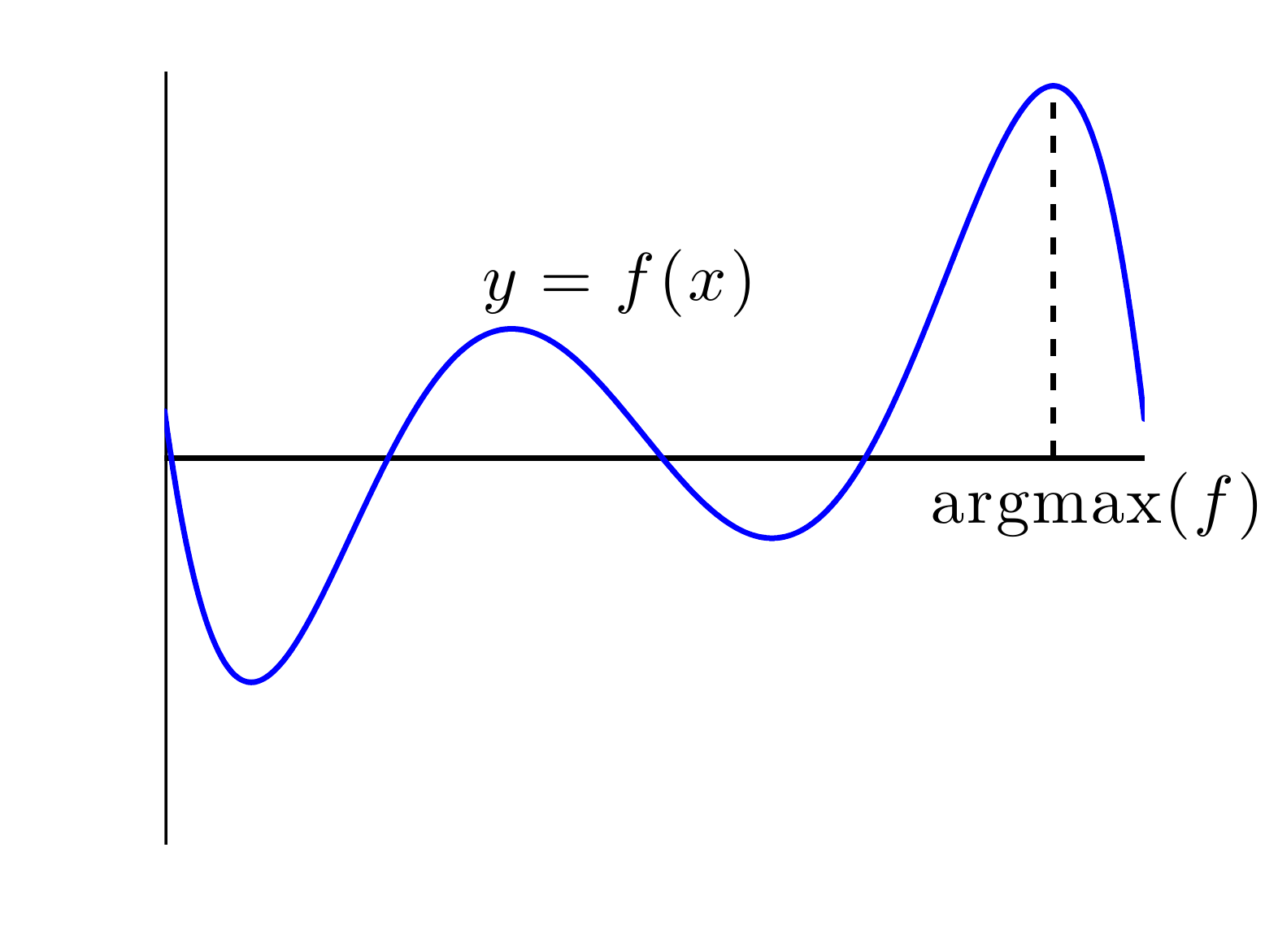}
\includegraphics[width=5cm]{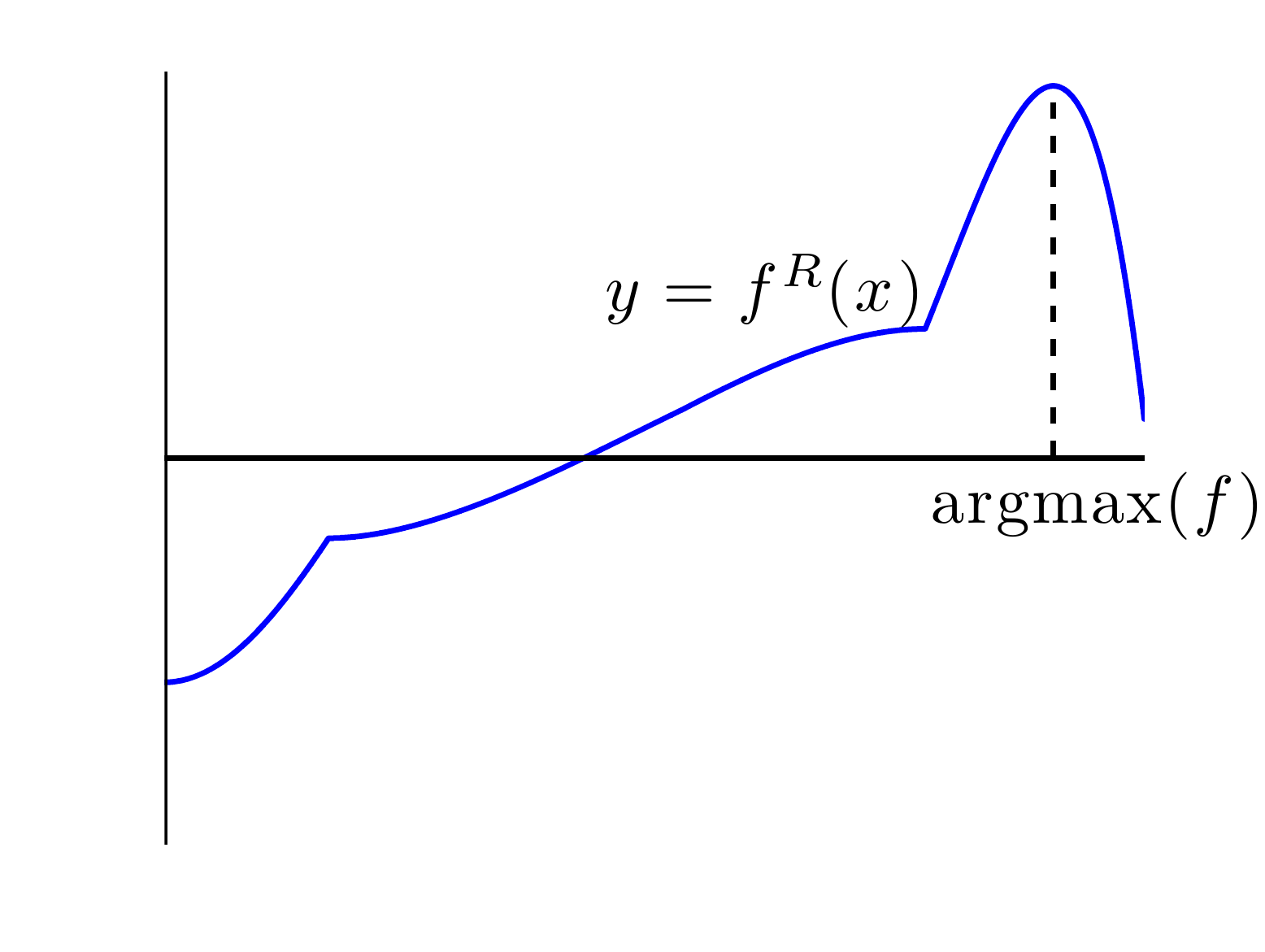}
\caption{Illustration of the rearrangement procedure: graph of a function $f$ (left) and graph of its rearrangement $f^R$ (right)} \label{illustrationSym}
\end{center}
\end{figure}

Observe that, by the equimeasurability property of monotone rearrangements the intervals $(0,y^+)$ and $(y^+,c(1))$, one has
\begin{equation}\label{eq:rennes1658}
\int_0^{c(1)}\tilde{m}^R(y)e^{\alpha \tilde{m}^R(y)}\, dy=\int_0^{c(1)}\tilde{m}(y)e^{\alpha \tilde{m} (y)}\, dy\leq -m_0.
\end{equation}

\paragraph{Step 2. The rearranged function $\tilde{m}^R$ decreases the Rayleigh quotient} Let us now show that~$u^R$ decreases the previous Rayleigh quotient. First, one has by property of monotone rearrangements that $u^R$ is positive. Writing
\begin{eqnarray*}
\lefteqn{\int_0^{c(1)}{\tilde{m}^R}(y)e^{2\alpha {\tilde{m}^R}(y)}u^R(y)^2\, dy=}\\
&\phantom{aaaaaaaa} \int_0^{c(1)}({\tilde{m}^R}(y)e^{2\alpha {\tilde{m}^R}(y)}+e^{-2\alpha})u^R(y)^2\, dy-e^{-2\alpha}\int_0^{c(1)}u^R(y)^2\, dy
\end{eqnarray*}
to deal with a positive weight and combining the Hardy-Littlewood inequality with the equimeasurability property of monotone rearrangements on $(0,y^+)$ and then on~$(y^+,c(1))$, we obtain
\begin{equation*}\label{1020eq2}
D_1\leq \int_0^{y^+}{\tilde{m}^R}(y)e^{2\alpha {\tilde{m}^R}(y)}u^R(y)^2\, dy\quad \text{and}\quad D_2\leq \int_{y^+}^{c(1)}{\tilde{m}^R}(y)e^{2\alpha {\tilde{m}^R}(y)}u^R(y)^2\, dy
\end{equation*}
and therefore
\begin{eqnarray}\label{1020eq1}
\int_0^{c(1)}{\tilde{m}^R}(y)e^{2\alpha {\tilde{m}^R}(y)}u^R(y)^2\, dy &\geq & \int_0^{c(1)}\tilde{m}(y)e^{2\alpha\tilde{m}(y)}u(y)^2\, dy\\
&=& \int_0^1 m(x) e^{\alpha m(x)}\ph(x)^2\, dx>0.\notag
\end{eqnarray}
Indeed, we used here that the function $\eta \mapsto \eta e^{2\alpha \eta}$ is increasing on $[-1,\kappa]$ whenever~$\alpha\leq -1/2$. Therefore, we claim that the rearrangement of the function $\tilde{m}e^{2\alpha\tilde{m}}$ according to the method described above coincides with the function $\tilde{m}^Re^{2\alpha {\tilde{m}^R}}$, whence the inequality above. Roughly speaking, we will use this inequality to construct an admissible test function in the Rayleigh quotient~\eqref{def:RRayleigh} from the knowledge of $u^R$.   

Also, we easily see that 
\begin{equation}\label{1020eq3}
(u^R)^2(0)=\min_{[0,y^+]}(u^R)^2\leq u^2(0)\;\textrm{ and }\;(u^R)^2(c(1))=\min_{[y^+,c(1)]}(u^R)^2\leq u^2(c(1)).
\end{equation}


Using now Poly\`a's inequality twice provides
\begin{equation}\label{1020eq4}
N_1\geq \int_0^{y^+} ((u^R)')^2+\beta (u^R)^2(0)\quad\text{ and }\quad N_2\geq \int_{y^+}^{c(1)} ((u^R)')^2+\beta (u^R)^2(c(1))
\end{equation}
As a result, by combining Inequalities \eqref{1020eq1}, \eqref{1020eq3} and \eqref{1020eq4}, one gets
\begin{equation}\label{1055ineqlamb}
\lambda_1^\beta(m) \geq \frac{ \int_0^{c(1)}(u^R)'(y)^2\, dy+\beta u^R(0)^2+\beta u^R(c(1))^2}{\int_0^{c(1)}{\tilde{m}^R}(y)e^{2\alpha {\tilde{m}^R}(y)}u^R(y)^2\, dy} . 
\end{equation}

Consider now the change of variable
$z=\int_0^y e^{\alpha \tilde{m}^R(t)}\, dt$,
as well as the functions~$m^R$ and $\varphi^R$ defined by
\begin{equation}\label{def:mRphiR}
m^R(z)=\tilde{m}^R(y)\quad \text{and}\quad \varphi^R(z)=u^R(y),
\end{equation}
for all $y\in [0,c(1)]$ and $z\in [0,1]$\footnote{Indeed, notice that, according to the equimeasurability property of monotone rearrangements, one has
$$
\int_0^{c(1)}e^{\alpha \tilde{m}^R(y)}\, dy=\int_0^{c(1)}e^{\alpha \tilde{m} (y)}\, dy=\int_0^1\, dx=1.
$$}.

Observe that $m^R$ is admissible for the optimal design problem \eqref{defSOP}. Indeed
$$
\int_0^1 m^R(z)\, dz = \int_0^{c(1)} \tilde{m}^R(y)e^{\alpha \tilde{m}^R(y)}\, dy \leq -m_0
$$
by \eqref{eq:rennes1658}. Since it is obvious that $-1\leq \tilde m^R\leq \kappa$ and that $\tilde m^R$ changes sign, we deduce immediately that $m^R$ satisfies Assumptions \eqref{assumption1} and \eqref{assumption2}.

Note that one has also $u^R(0)=\varphi^R(0)$, $u^R(c(1))=\varphi^R(1)$ and 
\begin{eqnarray*}
\int_0^{c(1)}(u^R)'(y)^2\, dy&=&\int_0^1 e^{\alpha m^R(z)}({\varphi^R}')^2(z)\, dz, \\
\int_0^{c(1)}{\tilde{m}^R}(y)e^{2\alpha {\tilde{m}^R}(y)}u^R(y)^2\, dy&=& \int_0^1m^R(z)e^{\alpha m^R(z)}\varphi^R(z)^2\, dz.
\end{eqnarray*}
In particular and according to \eqref{1020eq1} and the standard properties of rearrangements, there holds 
$
\int_0^1m^R(z)e^{\alpha m^R(z)}\varphi^R(z)^2\, dz>0,
$ and $\varphi^R\in \mathrm{H}^1(0,1)$
so that the function $\varphi^R$ is admissible in the Rayleigh quotient $\Re_{m^R}^\beta$. Hence, we infer from \eqref{1055ineqlamb} that
$
\lambda_1^\beta(m) \geq \Re_{m^R}^\beta [\varphi^R] \geq \lambda_1^\beta(m^R).
$

Finally, investigating the equality case of Poly\`a's inequality, it follows that the inequality \eqref{1055ineqlamb} is strict if $u$ is not unimodal, that is, if $\varphi$ is not unimodal (see for example \cite{BLR} and references therein). 

We have then proved the following result.

\begin{lemma}\label{lem:metz1808}
Every solution $m_*^\beta$ of the optimal design problem \eqref{defSOP} is unimodal, in other words, there exists $x_*$ such that $m_*^\beta$ is non-decreasing on $(0,x_*)$ and non-increasing on $(x_*,1)$. Moreover, the associated eigenfunction $\varphi_*^\beta$, \textit{i.e.} the solution of System \eqref{EVP} with $m=m_*^\beta$, is non-decreasing on $(0,x_*)$ and non-increasing on~$(x_*,1)$.
\end{lemma}

\begin{remark}\label{rk:largelambda1}
By using the change of variable \eqref{cdv} as well as the same reasonings and notations as above, it is notable that for every $m\in \cM$, the principal eigenvalue $\lambda_1^\beta(m)$ solves the eigenvalue problem
$$
-u''(y)=\lambda_1^\beta(m)\tilde{m}(y)e^{2\alpha \tilde{m}(y)}u(y), \quad \text{on }(0,c(1)).
$$
An easy but important consequence of this remark is the following: applying the Krein-Rutman theory to this problem yields existence, uniqueness and simplicity of $\lambda_1^\beta(m)$.
\end{remark}

\subsection{Existence of minimizers}

We start by stating and proving a Poincar\'e type inequality. The proof of the existence of a solution for the optimal design problem~\eqref{defSOP} relies mainly on Lemmas \ref{lem:metz1808} and \ref{poincare}.

\begin{lemma}\label{poincare} (Poincar\'e type inequality).~ Assume that $\alpha \in [0,\min\{1/2,\bar \alpha\})$. There exists a constant $C>0$ such that for every $\ph \in \mathrm{H}^1(0,1)$ and $m\in \cM$ satisfying $\int_0^1 m e^{\alpha m} \ph^2 >0$, one has $\int_0^1 \ph^2 \leq C \int_0^1 {\ph'}^2$.
\end{lemma} 

\begin{proof}
Assume that the inequality does not hold. Therefore, for every $k \in \N$, there exist $\ph_k\in \mathrm{H}^1(0,1)$ and $m_k \in \cM$ such that $\int_0^1 m_k e^{\alpha m_k} {\ph_k}^2>0$ and
\begin{equation}\label{eq:rennes1437}
	\int_0^1 {\ph_k}^2 > k \int_0^1 {{\ph_k}'}^2.
\end{equation}

First notice that we may assume that the functions $m_k$ and $\ph_k$ are non-increasing in $(0,1)$. Indeed, if we introduce the monotone non-increasing rearrangements $m_k^\searrow$ and $\ph_k^\searrow$ of $m_k$  and $\ph_k$, we have
$$\int_0^1 m_k e^{\alpha m_k} {\ph_k}^2 = \int_0^1 (m_k e^{\alpha m_k}+e^{-\alpha}) {\ph_k}^2 - \int_0^1 e^{-\alpha} {\ph_k}^2 \leq \int_0^1 m_k^\searrow e^{\alpha m_k^\searrow} {\ph_k^\searrow}^2,$$
where we have used the Hardy-Littlewood inequality and the equimeasurability property of the monotone rearrangements. Note that since the function $\eta \mapsto \eta e^{\alpha \eta}$ is increasing on $[-1,\kappa]$ whenever $\alpha \leq 1/2$, one has $(m_k e^{\alpha m_k})^\searrow = m_k^\searrow e^{\alpha m_k^\searrow}$.
Moreover,~\eqref{eq:rennes1437} implies that $\int_0^1 {\ph_k^\searrow}^2 > k \int_0^1 {{\ph_k^\searrow}'}^2$, where we have used the equimeasurability property and the Poly\'a inequality.

We may further assume that for each $k$, $\int_0^1 {\ph_k}^2=1$. Since the sequence $\ph_k$ is bounded in $\mathrm{H}^1(0,1)$, there is a subsequence $\ph_k$ such that $\ph_k \rightharpoonup \ph$ weakly in $\mathrm{H}^1$ and~$\ph_k \to \ph$ strongly in $\mathrm{L}^2$. As a consequence, $\int_0^1 {\ph'}^2 \leq \liminf \int_0^1 {{\ph_k}'}^2 = 0$, which implies that $\ph$ is contant in $(0,1)$. Note that since $\int_0^1 \ph^2=1$, $\ph$ must be positive in~$(0,1)$.

Since the functions $m_k$ are non-increasing, Helly's selection theorem ensures that, up to an extraction, $m_k$ converges pointwise to a function $m$. We infer that $-1\leq m\leq \kappa$ a.e., and that $\int_0^1 m \leq -m_0$, by dominated convergence. We also obtain that~$m_k e^{\alpha m_k} \to m e^{\alpha m}$ in $\mathrm{L}^2$. A consequence is that $\int_0^1 m_k e^{\alpha m_k} {\ph_k}^2 \to \int_0^1 m e^{\alpha m} \ph^2$ as $k\to \infty$. Indeed,
$\int_0^1 m e^{\alpha m} \ph^2 - \int_0^1 m_k e^{\alpha m_k} {\ph_k}^2
= \int_0^1 (m e^{\alpha m}-m_k e^{\alpha m_k}) \ph^2
+ \int_0^1 m_k e^{\alpha m_k} (\ph^2-{\ph_k}^2) 
~\underset{k\to \infty}{\longrightarrow} 0.$

Since $\ph$ is constant, we deduce that $\int_0^1 m e^{\alpha m}\geq 0$ . We also have that $\int_0^1 m e^{\alpha m}\leq 0$ since $\int_0^1 m_k e^{\alpha m_k}<0$ for every $k$ (recall that the inequality holds true for every function in $\cM$ whenever $\alpha <\bar \alpha$). We have finally proved that $\int_0^1 m e^{\alpha m} = 0$.
\\
We claim that $m$ cannot change sign. Indeed, otherwise, $m$ would lie in the set $\cM$, which would imply that $\int_0^1 m e^{\alpha m} <0$. We then deduce that $m=0$ a.e. in $(0,1)$, which is impossible since $\int_0^1 m\leq -m_0<0$.
\end{proof}

\begin{theorem}
If $\beta<+\infty$ (Neumann and Robin cases) and $\alpha \in [0,\min\{1/2,\bar \alpha\})$ (resp. $\beta=+\infty$  (Dirichlet case) and $\alpha \in [0,1/2)$), then the infimum $\lambda_*^\beta$ of $\lambda_1^\beta$ over~$\cM$ is achieved at some $m_*^\beta \in \cM$.
\end{theorem}

\begin{proof}
In this proof, we only deal with the case where $\beta<+\infty$. Indeed, we claim that all the lines can be easily adapted in the Dirichlet case since, in this case, the Poincar\'e inequality is satisfied without the assumption $\alpha<\bar{\alpha}$.

Consider a minimizing sequence $m_k$ for Problem \eqref{defSOP}.
By Lemma \ref{lem:metz1808}, one can assume that the functions $m_k$ are unimodal. 
As in the proof of Lemma \ref{poincare}, we may assume that $m_k$ converges pointwise\footnote{Indeed, the proof of Helly's selection theorem extends easily to the case of unimodal functions.}, and in $\mathrm{L}^2$ to a function $m^*\in \mathrm{L}^\infty(0,1)$. Moreover, $m^*$ satisfies $-1\leq m^*\leq \kappa$ a.e. in~$(0,1)$, and $\int_0^1 m^* \leq -m_0$.

For each $k$, let $\ph_k$ be the eigenfunction associated to $\lambda_1^\beta(m_k)$ with $\ph_k>0$. That is, the functions $\ph_k$ satisfy the variational formulation: for all $\psi\in \mathrm{H}^1(0,1)$,
\begin{equation}\label{eq:rennes0909}
	\int_0^1 e^{\alpha m_k} {\ph_k}' \psi' + \beta (\ph_k(0)\psi(0)+\ph_k(1)\psi(1))
		= \lambda_1^\beta(m_k) \int_0^1 m_k e^{\alpha m_k} \ph_k \psi.
\end{equation}
We may assume that for each $k$, $\int_0^1 m_k e^{\alpha m_k} {\ph_k}^2 =1$, which implies that $\lambda_1^\beta(m_k) = \int_0^1 e^{\alpha m_k} {{\ph_k}'}^2 +\beta(\ph_k(0)^2+\ph_k(1)^2)$. We deduce from Lemma \ref{poincare} that the sequence $\ph_k$ is bounded in $\mathrm{H}^1(0,1)$. Hence, there is a subsequence $\ph_k$ such that $\ph_k \to \ph$ in~$\mathrm{L}^2$, and~$\ph_k \rightharpoonup \ph$ in $\mathrm{H}^1$. By Lemma \ref{lem:metz1808}, we can assume that the functions $\ph_k$ are unimodal.
Write $\psi_k = \ph_k-\ph$. Taking $\psi=\psi_k$ in \eqref{eq:rennes0909} yields
\begin{eqnarray}
	\int_0^1 e^{\alpha m_k} {{\psi_k}'}^2 
&=&	- \int_0^1 e^{\alpha m_k}  {\psi_k}' \ph' 
	-  \beta \big(\psi_k(0)(\psi_k(0)+\ph(0)) + \psi_k(1)(\psi_k(1)+\ph(1) \big) \notag \\
&& 	+\; \lambda_1^\beta(m_k)
		\bigg(\int_0^1 m_k e^{\alpha m_k} {\psi_k}^2
		+ \int_0^1 m_k e^{\alpha m_k} \psi_k \ph \bigg).  \label{eq:rennes0922}
\end{eqnarray} 
Since $\ph_k \rightharpoonup \ph$ in $\mathrm{H}^1(0,1)$, one has $\psi_k(0)\to 0$ and $\psi_k(1)\to 0$ as $k\to \infty$. Therefore,~\eqref{eq:rennes0922} implies that $e^{-\alpha} \int_0^1 {{\psi_k}'}^2 \leq \int_0^1 e^{\alpha m_k} {{\psi_k}'}^2 \to 0$ as $k\to \infty$. As a consequence, the sequence $\ph_k$ converges in fact strongly to the function $\ph$ in $\mathrm{H}^1$.

As in the proof of Lemma \ref{poincare}, one has $\int_0^1 m^* e^{\alpha m^*} \ph^2=\lim_{k\to \infty} \int_0^1 m_k e^{\alpha m_k} {\ph_k}^2 = 1$. Firstly, this forces $m^*$ to change sign, which implies that $m^*\in \cM$. Secondly, one has
$$\lambda_1^\beta(m^*) \leq \int_0^1 e^{\alpha m^*} {\ph'}^2 + \beta(\ph(0)^2+\ph(1)^2) = \lim_{k\to \infty} \lambda_1^\beta(m_k).$$
Therefore, the infimum $\lambda_*^\beta$ is attained at $m^*\in \cM$.
\end{proof}

\subsection{Every minimizer is \textit{bang-bang}}\label{subsec:bang-bang}

At this step, we know according to Lemma \ref{lem:metz1808} that any minimizer $m_*^\beta$ is unimodal. Let us show moreover that it is {\it bang-bang}, in other words equal to $-1$ or $\kappa$ a.e. in $[0,1]$. 
\paragraph{Step 1. A new optimal design problem} The key point of the proof is the following remark: the function $m_*^\beta$ solves the optimal design problem
\begin{equation}\label{odp:m2317}
\inf_{\substack{m\in \cM\\ \int_0^1 m(x) e^{\alpha m(x)}\ph_*^\beta(x)^2\, dx>0}}\Re_m^\beta[\varphi_*^\beta]
\end{equation}
where $\ph_*^\beta$ denotes the eigenfunction associated to $\lambda_*^\beta = \lambda_1^\beta(m_*^\beta)$. Indeed, assume by contradiction the existence of $m\in \cM$ such that $\int_0^1 m(x) e^{\alpha m(x)}\ph_*^\beta(x)^2\, dx>0$ and $\Re_{m}^\beta[\varphi_*^\beta]< \Re_{m_*^\beta}^\beta[\varphi_*^\beta]$. This would hence imply that $\lambda_*^\beta >\lambda_1^\beta (m)$ whence the contradiction. Notice that this also implies in particular the existence of a solution for Problem \eqref{odp:m2317} and therefore that the constraint $\int_0^1 m(x) e^{\alpha m(x)}\ph_*^\beta(x)^2\, dx>0$ is not active at $m=m_*^\beta$. In other words, $\int_0^1 m_*^\beta (x) e^{\alpha m_*^\beta (x)}\ph_*^\beta(x)^2\, dx>0.$

Let us now introduce the set given by $\mathcal{I}=(0,1)\backslash (\{m_*^\beta=-1\}\cup\{m_*^\beta=\kappa\})$. Note that $\mathcal{I}$ is an element of the class of subsets of~$[0,1]$ in which $-1<m_*^\beta(x)<\kappa$ a.e. Notice that $\mathcal{I}$ also writes
$$
\mathcal{I}= \bigcup_{k=1}^{+\infty}\mathcal{I}_{ k}\quad \text{where }\mathcal{I}_{ k}=\left\{x\in (0,1) \mid -1+\frac{1}{k}< m_*^\beta(x)<\kappa-\frac{1}{k}\right\}.
$$
We will prove that the set $\mathcal{I}$ has zero Lebesgue measure. To this end, we argue by contradiction: we assume in the following of the proof that $|\mathcal{I}|>0$.

\paragraph{Step 2. The range of $m_*^\beta$ lies in $\{-1,0,\kappa\}$} 
In this step of the proof, we will prove that, up to a zero Lebesgue measure set, $\text{range}(m_*^\beta)\subset \{-1,0,\kappa\}$. To see this, we will use the previous remark and write the first order optimality conditions for Problem~\eqref{odp:m2317}. For that purpose, let us introduce the Lagrangian functional~$\mathcal{L}$ associated to Problem \eqref{odp:m2317}, defined by
$$
\mathcal{L}:\cM\times \R\ni (m,\eta)\mapsto \Re_{m}^\beta [\varphi_*^\beta]+\eta \left(\int_0^1m(x)\, dx+m_0\right).
$$
Note that we do not take into account the inequality constraint in the definition of the Lagrangian functional. Indeed, we aim at writing the first order optimality conditions at $m=m_*^\beta$ and we know that the inequality constraint is not active, according to the remark above. In the following, we will denote by $\eta^*$ the Lagrange multiplier associated to the (integral) equality constraint for Problem \eqref{odp:m2317}. In particular, $m_*^\beta$ minimizes the functional $\cM\ni m\mapsto \mathcal{L}(m,\eta^*)$. Notice that since we are dealing with inequality constraints, one has necessarily $\eta^*\geq 0$.

Since $|\mathcal{I}|>0$ by assumption, $\mathcal{I}_{k}$ is of positive measure when $k$ is large enough. If $|\mathcal{I}_k|>0$, take $x_0\in  \mathcal{I}_{k}$ and let $(G_{k,n})_{n\in \N}$ be a sequence of measurable subsets with $G_{n,k}$ included in $\mathcal{I}_{k}$ and containing $x_0$. Choosing $h=\1_{G_{k,n}}$, note that $m_*^\beta +th\in \cM$ and $m_*^\beta-th\in \cM$ when $t$ small enough. Writing 
$\mathcal{L}(m_*^\beta \pm th,\eta^* )\geq \mathcal{L}(m_*^\beta ,\eta^* )$, dividing this inequality by $t$ and letting $t$ go to 0, it follows that 
\begin{equation}\label{m:0021}
\langle d_m\mathcal{L}(m_*^\beta,\eta^*),h\rangle=0.
\end{equation}
Moreover, one computes
\begin{multline*}
\langle d_m\mathcal{L}(m_*^\beta,\eta^*),h\rangle \\
=  \int_{G_{n,k}} \frac{h(x) e^{\alpha m_*^\beta (x)}\big(\alpha {\varphi_*^\beta} '(x)^2-\lambda_*^\beta(\alpha m_*^\beta (x)+1)\varphi_*^\beta(x)^2\big)}{\int_0^1 m_*^\beta (x) e^{\alpha m_*^\beta (x)}\ph_*^\beta(x)^2\, dx}\, dx
+\eta^*|G_{n,k}|.
\end{multline*}
Assume without loss of generality that $\varphi_*^\beta$ is normalized such that $\int_0^1 m_*^\beta e^{\alpha m_*^\beta}{(\ph_*^\beta)}^2=1$.
Dividing the equality \eqref{m:0021} by $|G_{k,n}|$ and letting $G_{k,n}$ shrink to $\{x_0\}$ as $n\to +\infty$ shows that 
\begin{equation}\label{m0056}
\psi_0(x_0)=-\eta^*e^{-\alpha m_*^\beta (x_0)}\qquad \text{for almost every }x_0\in \mathcal{I}_{k},
\end{equation}
according to the Lebesgue Density Theorem, where
$$
\psi_0(x)=\alpha {\varphi_*^\beta} '(x)^2-\lambda_*^\beta(\alpha m_*^\beta (x)+1)\varphi_*^\beta(x)^2.
$$
\begin{lemma}\label{lemma:regm}
The set $\mathcal{I}$ (and therefore $\mathcal{I}_k$) is either an open interval or the union of two open intervals, and the restrictions of the functions $m_*^\beta$ and $\varphi_*^\beta$ to $\mathcal{I}$ belong to~$\mathrm{H}^2(\mathcal{I})$. 
\end{lemma}
\begin{proof}
The first point is obvious and results from the unimodal character of $m_*^\beta$ stated in Lemma~\ref{lem:metz1808}. Let us show that $m_*^\beta$ is continuous on each connected component of $\mathcal{I}$.

Let us consider the change of variable \eqref{cdv} introduced in Section \ref{sec:cdv}, namely $y=\int_0^x e^{-\alpha m_*^\beta (s)}\, ds$, for all $x\in [0,1]$. Introduce also the functions $c:[0,1]\ni x\mapsto \int_0^x e^{-\alpha m_*^\beta (s)}\, ds$ and $\tilde{m}_*^\beta$ defined on $[0,c(1)]$ by $\tilde{m}_*^\beta(y)=m_*^\beta(c^{-1}(y))$. The crucial argument rests upon the fact that
$
c(\mathcal{I})=c(\{-1<m_*^\beta<\kappa\})=\{-1<\tilde{m}_*^\beta<\kappa\},
$ 
since $c$ is in particular continuous. Furthermore, it follows from \eqref{m0056} that the function~$\tilde{m}_*^\beta$ satisfies 
\begin{equation}\label{trainnancy2223}
\alpha e^{-2\alpha \tilde{m}_*^\beta(y_0)}{u_*^\beta} '(y_0)^2-\lambda_*^\beta(\alpha \tilde{m}_*^\beta (y_0)+1)u_*^\beta(y_0)^2=-\eta^*e^{-\alpha \tilde{m}_*^\beta (y_0)}\quad \text{on }c(\mathcal{I}),
\end{equation}
where $u_*^\beta$ is defined by $u_*^\beta(y)=\varphi_*^\beta(c^{-1}(y))$ for all $y\in [0,c(1)]$. A simple computation shows that the function $u_*^\beta$ solves in a distributional sense the o.d.e.
$$
-{u_*^\beta}''(y)=\lambda_*^\beta \tilde{m}_*^\beta(y)e^{2\alpha \tilde{m}_*^\beta(y)}{u_*^\beta}(y)\quad \text{in }(0,c(1)].
$$
By using standard elliptic regularity arguments (see e.g. \cite{Brezis}), we infer that $u_*^\beta$ belongs to $\mathrm{H}^2(0,1)$ and is in particular $C^1$ on $[0,c(1)]$. According to \eqref{trainnancy2223} and applying the implicit functions theorem, we get that the function $\tilde{m}_*^\beta$ is necessarily itself $C^1$ on~$\mathcal{I}$. Using the regularity of $\tilde{m}_*^\beta$ and $c$, and since the derivative of $c$ is pointwisely bounded by below by $e^{-\alpha\kappa}$, we infer that the restriction of the function $m_*^\beta=\tilde{m}_*^\beta\circ c^{-1}$ to~$\mathcal{I}$ belongs to $\mathrm{H}^1(\mathcal{I})$.
Furthermore, consider one connected component, say $(x^1_{\mathcal{I}},x^2_{\mathcal{I}})$ of~$\mathcal{I}$. Since for all $x\in (x^1_{\mathcal{I}},x^2_{\mathcal{I}})$ there holds 
$c(x)=\int_0^x e^{-\alpha m_*^\beta(s)}\, ds$, one infers that for all~$x\in (x^1_{\mathcal{I}},x^2_{\mathcal{I}})$, one has $c'(x)=e^{-\alpha m_*^\beta(x)}$ and thus, $c\in \mathrm{H}^2(\mathcal{I})$ by using that $m_*^\beta \in \mathrm{H}^1(\mathcal{I})$. As a result, since $m_*^\beta=\tilde{m}_*^\beta\circ c^{-1}$, one gets successively that~$m_*^\beta$ and~$\varphi_*^\beta $ are $\mathrm{H}^2$ on $\mathcal{I}$ (by using in particular \eqref{EVP} for $\varphi_*^\beta$).
\end{proof}

According to Lemma \ref{lemma:regm}, the function $m_*^\beta$ is $\mathrm{H}^2$ on each interval of $\mathcal{I}$ (and hence of $\mathcal{I}_k$). Therefore, using that
$
{\varphi_*^\beta}''(x)=-\alpha {m_*^\beta}'{\varphi_*^\beta}'-\lambda_*^\beta m_*^\beta \varphi_*^\beta$ on $\mathcal{I}_k,
$
this last equality being understood in $\mathrm{L}^2(\mathcal{I}_k)$, one computes 
\begin{equation}\label{paris17H09}
\psi_0'(x) = -2{\varphi_*^\beta}'(x)\left(\lambda_*^\beta(2\alpha m_*^\beta+1)\varphi_*^\beta+\alpha^2{m_*^\beta}'(x){\varphi_*^\beta}'(x)\right)-\alpha \lambda_*^\beta {m_*^\beta}'(x)\varphi_*^\beta(x)^2
\end{equation}
for every $x\in\mathcal{I}_k$. According to Lemma \ref{lem:metz1808}, we claim that  ${\varphi_*^\beta}'(x)$ and ${m_*^\beta}'(x)$ have the same sign (with the convention that the number 0 is at the same time of positive and negative sign) for a.e. $x\in (0,1)$ and therefore ${m_*^\beta}'(x){\varphi_*^\beta}'(x)\geq 0$ for a.e. $x\in (0,1)$. Since $\alpha \leq 1/2$, one has $2\alpha m_*^\beta+1\geq -2\alpha+1\geq 0$ and with the notations of Lemma \ref{lem:metz1808}, it follows that $\psi_0'(x)$ is nonpositive on $(0,x_*)$ and nonnegative on $(x_*,1)$, implying that~$\psi_0$ is non-increasing on $(0,x_*)$ and non-decreasing in $(x_*,1)$. Moreover and according to the previous discussion, it is obvious that $-\eta^*e^{-\alpha m_*^\beta}$ is non-decreasing on $(0,x_*)$ and non-increasing in $(x_*,1)$, since $\eta^*\geq 0$. 

We then infer from the previous reasoning and since the integer $k$ was chosen arbitrarily that there exist $x_0$, $y_0$, $x_1$, $y_1$ such that $\mathcal{I}=(x_0,y_0)\cup (x_1,y_1)$ with $0<x_0\leq y_0\leq x_*\leq x_1\leq y_1$ and the equality
\begin{equation}\label{paris16H49}
\psi_0(x)=-\eta^*e^{-\alpha m_*^\beta (x)}
\end{equation}
holds true on $\mathcal{I}$. If $x_0<y_0$ (resp. $x_1<y_1$), notice that one has necessarily $m_*^\beta=0$ on~$(x_0,y_0)$ (resp. on $(x_1,y_1)$). Indeed, it follows from \eqref{paris16H49} and the monotonicity properties of $\psi_0$ and $m_*^\beta$ on $(x_0,y_0)$ and $(x_1,y_1)$ that $\psi_0$ and $m_*^\beta$ are constant on~$(x_0,y_0)$ and~$(x_1,y_1)$. According to \eqref{paris17H09} and since $\alpha \in [0,1/2]$, it follows that $\varphi_*^\beta$ is also constant on $(x_0,y_0)$ and $(x_1,y_1)$. Using the equation solved by $\varphi_*^\beta$, one shows that necessarily, $m_*^\beta=0$ on $(x_0,y_0)$ and $(x_1,y_1)$. This achieves the proof that $m_*^\beta$ is equal to $-1$, $0$ or $\kappa$ a.e. 

\paragraph{Step 4. The minimizer $m_*^\beta$ is bang-bang} 
In this last step, we will use the second order optimality conditions to reach a contradiction.
Since by hypothesis, $\mathcal{I}$ has positive Lebesgue measure, it is not restrictive to assume that $x_0<y_0$. We will reach a contradiction with an argument using the second order optimality conditions.
Introduce the functional $\Re :\cM \ni m\mapsto \Re_m^\beta[\varphi_*^\beta]$ as well as its first and second order derivative in an admissible direction $h$ denoted respectively $\langle d\Re(m),h\rangle$ and~$d^2\Re (m)(h,h)$. One has 
$$
\langle d\Re(m_*^\beta),h\rangle= \alpha \int_0^1 h e^{\alpha m_*^\beta} ({\varphi_*^\beta}')^2-\lambda_*^\beta \int_0^1 h(1+\alpha m_*^\beta)e^{\alpha m_*^\beta}{(\varphi_*^\beta)}^2.
$$
Consider an admissible\footnote{For every $m\in \cM$, the tangent cone to the set $\cM$ at $m$, denoted by $\mathcal{T}_{m,\cM}$ is the set of functions $h\in \mathrm{L}^\infty(0,1)$ such that, for any sequence of positive real numbers $\varepsilon_n$ decreasing to $0$, there exists a sequence of functions $h_n\in \mathrm{L}^\infty(0,1)$ converging to $h$ as $n\rightarrow +\infty$, and $m+\varepsilon_nh_n\in\cM$ for every $n\in\N$ (see for instance \cite[chapter 7]{henrot-pierre}).\label{footnote:cone}} perturbation $h$ supported by $(x_0,y_0)$. The first order optimality conditions yield that $\langle d\Re(m_*^\beta),h\rangle=0$ and one has therefore
\begin{eqnarray*}
d^2\Re (m_*^\beta)(h,h)&=& \alpha^2 \int_0^1 h^2 e^{\alpha m_*^\beta} ({\varphi_*^\beta}')^2-\alpha \lambda_*^\beta \int_0^1h^2 e^{\alpha m_*^\beta} (2+\alpha m_*^\beta)({\varphi_*^\beta})^2\\
&=&  -2\alpha \lambda_*^\beta \int_0^1h^2 ({\varphi_*^\beta})^2<0
\end{eqnarray*}
whenever $\int h^2>0$, since $\ph_*^\beta$ is constant and $m_*^\beta=0$ on $(x_0,y_0)$. It follows that for a given admissible perturbation $h$ as above, we have
$
\Re(m_*^\beta+\varepsilon h)<\Re(m_*^\beta)
$
provided that $\varepsilon>0$ is small enough.  We have reached a contradiction, which implies that $x_i=y_i$, $i=0,1$. 

We have then proved the following lemma.

\begin{lemma}\label{lem:metz1816}
{Every solution $m_*^\beta$ of the optimal design problem \eqref{defSOP} is {\it bang-bang}, in other words equal to $-1$ or $\kappa$ a.e. in $(0,1)$.}
\end{lemma}

\paragraph{Proof of Theorem~\ref{thm:rennes1110}} We end this section with providing the proof of Theorem~\ref{thm:rennes1110}. Assume that $0\leq \alpha < \min \{1/2,\bar \alpha \}$ and consider a solution $m_*^\beta$ of the optimal design problem \eqref{defSOP}. Introduce the principal eigenfunction $\ph_*^\beta$ associated with $m_*^\beta$, normalized in such a way that $\int_0^1 m_*^\beta e^{\alpha m_*^\beta} {(\ph_*^\beta)}^2 = 1$.

By Lemmas \ref{lem:metz1808} and \ref{lem:metz1816}, the function $m_*^\beta$ is unimodal and \textit{bang-bang}. On can easily construct a sequence of smooth functions $m_k$ in $\cM$ such that $m_k$ converges a.e. to $m_*^\beta$ in $(0,1)$. The dominated convergence theorem yields that $\int_0^1 m_k e^{\alpha m_k} {(\ph_*^\beta)}^2 \to \int_0^1 m_*^\beta e^{\alpha m_*^\beta} {(\ph_*^\beta)}^2=1$ as $k\to \infty$. Hence, the following inequality holds when $k$ is large enough
$$\lambda_1^\beta(m_k) \leq \Re_{m_k}^\beta[\ph_*^\beta]  \underset{k\to \infty}{\longrightarrow} \Re_{m_*^\beta}^\beta[\ph_*^\beta] = \lambda_1^\beta(m_*^\beta),$$
by dominated convergence. We deduce that $\limsup \lambda_1^\beta(m_k) \leq \lambda_1^\beta(m_*^\beta)$, which yields that $\lambda_1^\beta(m_k) \to \lambda_1^\beta(m_*^\beta)$ as $k\to \infty$ and proves \eqref{eq:rennes1128}. As a consequence, Lemma \ref{lem:metz1816} implies that $\lambda_1^\beta$ does not reach its infimum over $\cM\cap C^2(\overline{\Omega})$.

\subsection{Conclusion: end of the proof}\label{sec:cclproof}
According to Lemmas \ref{lem:metz1808} and \ref{lem:metz1816}, any minimizer $m_*^\beta$ for the optimal design problem \eqref{defSOP} is unimodal and {\it bang-bang}. We then infer that it remains to investigate the case where the admissible design $m$ writes 
$
m_*^\beta=(\kappa+1)\1_{I}-1,
$
where $I$ is a subinterval of $(0,1)$ whose length is $\delta = \frac{1-\widetilde{m}_{0}}{\kappa+1}$ with $\widetilde{m}_0\leq -m_0$ (note that~$-\widetilde{m}_0$ plays the role of the optimal amount of resources $\int_0^1m^*$ for $m^*$ solving Problem \eqref{EVPlambdastar}). This is the main goal of this section.

For that purpose, let us introduce the optimal design problem
\begin{equation}\label{pbm0038}
\inf\{\lambda_1^\beta (m), \ m=(\kappa+1)\1_{(\xi,\xi+\delta)}-1, \ \xi\in [0, (1-\delta)/2]\}.
\end{equation}
\begin{remark}
In the formulation of the problem above, we used an easy symmetry argument allowing to reduce the search of $\xi$ to the interval $[0, (1-\delta)/2]$ instead of $[0,1-\delta]$.
\end{remark}

The following propositions conclude the proof of Theorems \ref{thm:mainDN}, \ref{thm:main} and~\ref{thm:mainD}. Their proofs are given respectively in Appendices \ref{sec:appendoptlocint} and \ref{sec:prooftheo:compm0} below.

\begin{proposition}\label{prop:interval}
Let $\kappa>0$, $\beta\geq 0$, $\alpha\in [0,\bar\alpha)$, $\widetilde{m}_0\in [m_0,1)$ and $\delta$ be defined as above. The optimal design problem \eqref{pbm0038} has a solution. Moreover,
\begin{itemize}
\item if $\beta <\beta_{\alpha,\delta}$, then $m=(\kappa+1)\1_{(0,\delta)}-1$ and $m=(\kappa+1)\1_{(1-\delta,1)}-1$ are the only solutions of Problem \eqref{pbm0038},
\item if $\beta >\beta_{\alpha,\delta}$, then $m=(\kappa+1)\1_{((1-\delta)/2,(1+\delta)/2)}-1$ is the only solution of Problem \eqref{pbm0038},
\item if $\beta =\beta_{\alpha,\delta}$, then every function $m=(\kappa+1)\1_{(\xi,\xi+\delta)}-1$ with $\xi\in [0,1-\delta]$ solves Problem~\eqref{pbm0038}.
\end{itemize}
\end{proposition}

\begin{proposition}\label{prop:activeconstraint}
Under the assumptions of Proposition \ref{prop:interval}, if one assumes moreover that \eqref{eq:rennes1810} holds true in the case $\beta\geq \beta_{\alpha,\delta}$, then one has $\int_0^1 m_*^\beta=-m_0$.
\end{proposition}

\section{Perspectives}\label{sec:ccl}

The same issues as those investigated in this work remain relevant in the multi-dimensional case, from the biological as well as the mathematical point of view. Indeed, the same considerations as in Section \ref{Section:BiologicalModel} lead to investigate the problem 
$$
\inf_{m\in \cM}\lambda_1^\beta(m)\quad \text{with}\quad \lambda_1^\beta (m) = \inf_{\ph \in \cS} \frac {\int_\Omega e^{\alpha m} {|\nabla \ph|}^2 
		+ \beta \int_{\partial \Omega} \ph^2} {\int_\Omega m e^{\alpha m} \ph^2}.
$$
Such a problem needs a very careful analysis. It is likely that such analysis will strongly differ from the one led in this article. Indeed, we claim that except maybe for some particular sets $\Omega$ enjoying symmetry properties, we cannot use directly the same kind of rearrangement/symmetrization techniques. 

Furthermore, the change of variable introduced in Section \ref{sec:cdv} is proper to the study of Sturm-Liouville equations. We used it to characterize persistence properties of the diffusive logistic equation with an advection term and to exploit the first and second order optimality conditions of the optimal design problem above, but such a rewriting has {\it a priori} no equivalent in higher dimensions.

We plan to investigate the following issues:
\begin{itemize}
\item {\it (biological model)} existence, simplicity of a principal eigenvalue for weights $m$ in the class $\cM$, without additional regularity assumption;
\item {\it (biological model)} time asymptotic behavior of the solution of the logistic diffusive equation with an advection term, and characterization of the alternatives in terms of the principal eigenvalue;
\item {\it (optimal design problem)} existence and {\it bang-bang} properties of minimizers;  
\item {\it (optimal design problem)} development of a numerical approach to compute the minimizers.  
\end{itemize}
It is notable that, in the case where $\alpha=0$, several theoretical and numerical results gathered in \cite{KaoLouYanagida} suggest that properties of optimal shapes, whenever they exist, strongly depend on the value of $m_0$.

Another interesting issue (relevant as well in the one and multi-D models) concerns the sharpness of the smallness assumptions on $\alpha$ made in Theorems \ref{thm:mainDN}, \ref{thm:main} and~\ref{thm:mainD}. From these results, one is driven to wonder whether this assumption can be relaxed or even removed.


\appendix

\section{Sketch of the proof of Proposition \ref{pp-ev}}\label{proof-pp-ev}

In this appendix, we briefly sketch the proof of Proposition~\ref{pp-ev}, for sake of completeness. The proof follows a method proposed by Hess and Kato in~\cite{MR588690} (see also~\cite{MR1469392}).

We start by considering the eigenvalue problem
\begin{equation}\label{rennes-1810}
\left\{\begin{array}{ll}
	-\Div(e^{\alpha m} \nabla \ph)-\lambda m e^{\alpha m} \ph = \mu \ph & \text{in } \Omega,\\
	B\ph = 0 & \text{on } \partial \Omega ,
\end{array} \right.
\end{equation}
where $\lambda$ is a real number, and $B$ is defined by $B\ph=\ph$ in the case of Dirichlet boundary conditions, and $B\ph=e^{\alpha m} \partial_n \ph + \beta \ph$ in the case of Neumann or Robin conditions.
A standard application of Krein-Rutman theory implies that the eigenvalue problem~\eqref{rennes-1810} has a unique principal eigenvalue $\mu(\lambda)$. The eigenvalue $\mu(\lambda)$ is simple, and it is the smallest eigenvalue of Problem~\eqref{rennes-1810} (see for example~\cite{MR2205529}). As a consequence,~$\lambda$ is a principal eigenvalue of Problem \eqref{EVPMultiD} if and only if $\mu(\lambda)=0$.

It is also known that the principal eigenvalue $\mu(\lambda)$ can be characterized by
$$
	\mu(\lambda) = \inf~ \left\lbrace \int_{\Omega} e^{\alpha m} {|\nabla \ph|}^2 
		- \lambda \int_\Omega me^{\alpha m} \ph^2,~~ 
		\ph \in \mathrm{H}^1_0(\Omega),~ \int_\Omega \ph^2 = 1 \right\rbrace
$$
in the case of Dirichlet boundary conditions, and by
$$
	\mu(\lambda) = \inf~ \left\lbrace \int_{\Omega} e^{\alpha m} {|\nabla \ph|}^2 
		+ \beta \int_{\partial \Omega} \ph^2 - \lambda \int_\Omega me^{\alpha m} \ph^2,~~ 
		\ph \in \mathrm{H}^1(\Omega),~ \int_\Omega \ph^2 = 1\right\rbrace
$$
in the case of Neumann or Robin conditions.

Notice that since the function $\lambda \mapsto \mu(\lambda)$ is defined as an infimum of affine then concave functions of $\lambda$, it is itself concave. Moreover, considering well-chosen test functions in the Rayleigh quotient, we see that $\mu(\lambda)\to -\infty$ as $|\lambda|\to \infty$. Indeed, the assumption $\int_\Omega m<0$ ensures that there are admissible test functions $\ph_1$ and $\ph_2$ such that $\int_\Omega m e^{\alpha m}{\ph_1}^2 >0$ and $\int_\Omega m e^{\alpha m}{\ph_2}^2 < 0$.

If the boundary conditions are of Dirichlet type, or of Robin type with $\beta\not=0$, then it is obvious that $\mu(0)>0$. Therefore, the function $\lambda\mapsto \mu(\lambda)$ has exactly two zeros: one positive and one negative. As a consequence Problem \eqref{pp-ev} has a unique positive principal eigenvalue.

In the case of Neumann boundary conditions, that is when $\beta=0$, it is clear that $\mu(0)=0$. Moreover, differentiating $m$ with respect to $\lambda$ yields that $\mu'(\lambda) = -\big(\int_\Omega m e^{\alpha m} v^2\big) / \big(\int_\Omega v^2\big)$, where $v$ is any eigenfunction associated with the eigenvalue~$\mu(\lambda)$. As a consequence, $\mu'(0)= -\frac 1 {|\Omega|} \int_\Omega m e^{\alpha m}$, and we deduce that:
\begin{itemize}
\item if $\int_\Omega m e^{\alpha m}<0$, then there exists a unique positive principal eigenvalue;
\item if $\int_\Omega m e^{\alpha m}\geq 0$, then $0$ is the only non-negative principal eigenvalue.
\end{itemize}

\section{Optimal location of an interval (Proof of Proposition~\ref{prop:interval})}\label{sec:appendoptlocint}
This section is devoted to the proof of Proposition \ref{prop:interval}. For that purpose, let us assume that $m=(\kappa+1)\1_{(\xi,\xi+\delta)}-1$ with $\delta=\frac{1-\widetilde{m}_0}{\kappa+1}$. Notice that $\delta$ is chosen in such a way that $\int_0^1m=-\widetilde{m}_0$ and that one has necessarily $\xi\in [0,1-\delta]$. In what follows, we will restrict the range of values for $\xi$ to the interval $[0,(1-\delta)/2]$ by noting that 
\begin{equation}\label{sp0821}
\lambda_1^\beta (m)=\lambda_1^\beta (\hat m),\quad\textrm{ with }\hat m= (\kappa+1)\1_{(1-\xi-\delta,1-\xi)}-1.
\end{equation}

\paragraph{Step 1. Explicit solution of System~\eqref{EVP}} Assume temporarily that $\xi>0$. 
In that case, according to standard arguments of variational analysis, System~\eqref{EVP} becomes
\begin{equation}\label{metz:1752}
\left\{ \begin{array}{ll}
-\ph''=-\lambda \ph & \text{in }~ (0,\xi),\\[1mm]
-\ph''=\lambda \kappa \ph & \text{in }~ (\xi,\xi+\delta),\\[1mm]
-\ph''=-\lambda \ph & \text{in }~ (\xi+\delta,1),\\[1mm]
\ph(\xi^-)=\ph(\xi^+), \quad \ph((\xi+\delta)^-)=\ph(\xi+\delta)^+) , &\\[1mm]
e^{-\alpha}\ph'(0)=\beta \ph(0), \quad e^{-\alpha }\ph'(1)=-\beta \ph(1). &
\end{array}
\right.
\end{equation}
completed by the following jump conditions on the derivative of $\ph$
\begin{equation}\label{metz:1753}
e^{\alpha(\kappa+1)}\ph'(\xi^+)=\ph'(\xi^-), \quad\textrm{and}\quad \ph'((\xi+\delta)^+)=e^{\alpha(\kappa+1)}\ph'((\xi+\delta)^-).
\end{equation}
According to \eqref{metz:1752}, there exists a pair $(A,B)\in \R^2$ such that
\begin{equation}\label{noel1136}
\ph(x)=\left\{\begin{array}{ll}
A\frac{\sqrt{\lambda}\cosh(\sqrt{\lambda} x)+\beta e^\alpha \sinh(\sqrt{\lambda} x)}{\sqrt{\lambda} \cosh(\sqrt{\lambda}\xi)+\beta e^\alpha \sinh(\sqrt{\lambda}\xi)} & \text{in }~ (0,\xi),\\[1mm]
C\cos (\sqrt{\lambda \kappa}x)+D\sin (\sqrt{\lambda \kappa}x) & \text{in }~ (\xi,\xi+\delta),\\[1mm]
B\frac{\sqrt{\lambda}\cosh(\sqrt{\lambda} (x-1))-\beta e^{\alpha}\sinh(\sqrt{\lambda} (x-1))}{\sqrt{\lambda} \cosh(\sqrt{\lambda}(\xi+\delta-1))-\beta e^{\alpha} \sinh(\sqrt{\lambda}(\xi+\delta-1))} & \text{in }~ (\xi+\delta,1) ,
\end{array}
\right.
\end{equation}
where the expression of the constants $C$ and $D$ with respect to $A$ and $B$ is determined by using the continuity of $\ph$ at $x=\xi$ and $x=\xi+\delta$, namely
$$
C = \frac{A\sin (\sqrt{\lambda\kappa}(\xi+\delta))-B\sin (\sqrt{\lambda\kappa}\xi )}{\sin(\sqrt{\lambda\kappa}\delta)},~~~ 
D= -\frac{A\cos (\sqrt{\lambda\kappa}(\xi+\delta))-B\cos (\sqrt{\lambda\kappa}\xi )}{\sin(\sqrt{\lambda\kappa}\delta)} .
$$
Plugging \eqref{noel1136} into \eqref{metz:1753}, the jump condition \eqref{metz:1753} rewrites 
$$
M\begin{pmatrix}
A\\ B
\end{pmatrix}=\begin{pmatrix}
0\\ 0
\end{pmatrix}\quad \textrm{with}\quad M=\begin{pmatrix}
m_{11} & m_{12}\\ 
m_{21} & m_{22}
\end{pmatrix},
$$
where
\begin{eqnarray*}
m_{11} & = & \sqrt{\kappa}e^{\alpha(\kappa+1)}\left(\sqrt{\lambda} \cosh(\sqrt{\lambda}\xi)+\beta e^\alpha\sinh(\sqrt{\lambda}\xi)\right)\cos(\sqrt{\lambda \kappa}\delta)\\
&& \qquad +\left(\sqrt{\lambda} \sinh(\sqrt{\lambda}\xi)+\beta e^\alpha \cosh(\sqrt{\lambda}\xi)\right)\sin(\sqrt{\lambda \kappa}\delta) , \\
m_{12} & = & - \sqrt{\kappa}e^{\alpha(\kappa+1)} (\sqrt{\lambda} \cosh(\sqrt{\lambda}\xi)+\beta e^\alpha\sinh(\sqrt{\lambda}\xi)) , \\
m_{21} & = & -\sqrt{\kappa}e^{\alpha(\kappa+1)}\left(\sqrt{\lambda} \cosh(\sqrt{\lambda}(\xi+\delta-1))-\beta e^{\alpha}\sinh(\sqrt{\lambda}(\xi+\delta-1))\right) , \\
m_{22} & = & \sqrt{\kappa}e^{\alpha(\kappa+1)}\left(\sqrt{\lambda} \cosh(\sqrt{\lambda}(\xi+\delta-1))-\beta e^{\alpha}\sinh(\sqrt{\lambda}(\xi+\delta-1))\right)\cos (\sqrt{\lambda\kappa}\delta) \\
&& \qquad - \left(\sqrt{\lambda} \sinh(\sqrt{\lambda}(\xi+\delta-1))-\beta e^{\alpha} \cosh(\sqrt{\lambda}(\xi+\delta-1))\right)\sin (\sqrt{\lambda\kappa}\delta) .
\end{eqnarray*}

\paragraph{Step 2. A transcendental equation} Since the pair $(A,B)$ is necessarily nontrivial (else, the function $\ph$ would vanish identically which is impossible by definition of an eigenfunction), one has necessarily 
$
\det M=m_{11}m_{22}-m_{12}m_{21}=0.
$
This allows to obtain the so-called {\it transcendental equation}. After lengthly computations, this equation can be recast in the simpler form
\begin{equation}\label{metz22:29}
\sin (\sqrt{\lambda\kappa}\delta)F_{\alpha}(\xi,\beta,\lambda)=0,
\end{equation}
where
\begin{equation}\label{Fet0825}
F_{\alpha}(\xi,\beta,\lambda)= -F_{\alpha}^s(\xi,\beta,\lambda)\sin(\sqrt{\lambda\kappa}\delta) +\sqrt{\kappa}e^{\alpha(\kappa+1)}F^c_{\alpha}(\beta,\lambda)\cos (\sqrt{\lambda\kappa}\delta)
\end{equation}
with
\begin{eqnarray*}
F^s_{\alpha}(\xi,\beta,\lambda)&=& \beta e^\alpha\sqrt{\lambda}(\kappa e^{2 \alpha(\kappa+1)}-1) \sinh (\sqrt{\lambda}(1-\delta))\\
&& +\frac{1}{2}(1+\kappa e^{2\alpha (1+\kappa)})(\lambda-\beta^2e^{2\alpha})\cosh\left(\sqrt{\lambda}(1-2\xi-\delta)\right)\\
&& +\frac{1}{2}(\kappa e^{2\alpha (1+\kappa)}-1)(\beta^2e^{2\alpha}+\lambda)\cosh (\sqrt{\lambda}(1-\delta)),\\
F^c_{\alpha}(\beta,\lambda) &=& (\lambda+\beta^2e^{2\alpha})\sinh(\sqrt{\lambda}(1-\delta))+2\beta\sqrt{\lambda}e^\alpha\cosh(\sqrt{\lambda}(1-\delta)) .
\end{eqnarray*}

In the sequel, we will denote by $\lambda^\beta_*$ (resp. $m^\beta_*$) the minimal value for Problem~\eqref{pbm0038} (resp. a minimizer), \textit{i.e.}
$$
\lambda^\beta_*=\lambda_1^\beta (m^\beta_*)=\inf\{\lambda_1^\beta (m), \ m=(\kappa+1)\1_{(\xi,\xi+\delta)}-1, \ \xi\in [0, (1-\delta)/2]\}.
$$
The existence of such a pair follows from the continuity of $[0,1-\delta]\ni \xi\mapsto \lambda^\beta_1((\kappa+1)\1_{(\xi,\xi+\delta)}-1)$ combined with the compactness of $[0,(1-\delta)/2]$.
\begin{remark}\label{rk:0039}
In the Dirichlet case (corresponding formally to take $\beta=+\infty$) and for the particular choice $\xi=0$, the transcendental equation rewrites
$$
\tan(\sqrt{\lambda\kappa}\delta)=-\sqrt{\kappa}e^{\alpha(\kappa+1)}\tanh(\sqrt{\lambda}(1-\delta)).
$$
It is then easy to prove that the first positive root of this equation $\lambda_{D,0}$ is such that~$\sqrt{\lambda_{D,0}}\in (\pi/(2\sqrt{\kappa}\delta)),\pi/(\sqrt{\kappa}\delta)))$. We thus infer that
$$
\inf_{\substack{\xi\in [0, (1-\delta)/2]}} \lambda_1^\beta ((\kappa+1)\1_{(\xi,\xi+\delta)}-1)\leq \inf_{\xi\in [0, (1-\delta)/2]} \lim_{\beta\to +\infty}\lambda_1^\beta ((\kappa+1)\1_{(\xi,\xi+\delta)}-1) <\frac{\pi^2}{\kappa \delta^2},
$$
by noting that the mapping $\R_+\ni \beta \mapsto \lambda_1^\beta (m)$ is non-decreasing. Indeed, this monotonicity property follows from the fact that $\lambda_1^\beta (m)$ writes as the infimum of affine functions that are increasing with respect to $\beta$. As a consequence, there holds 
\begin{equation}\label{eq0040}
\sin \left(\sqrt{\lambda^\beta_*\kappa}\delta\right)> 0.
\end{equation}
\end{remark}

According to Remark \ref{rk:0039}, one can restrict the study to the parameters $\lambda$ such that \eqref{eq0040} holds true and in particular $\sin (\sqrt{\lambda\kappa}\delta)\neq 0$. Hence, the transcendental equation \eqref{metz22:29} simplifies into
\begin{equation}\label{et0822} 
F_{\alpha} (\xi,\beta,\lambda)=0.
\end{equation}
A standard application of the implicit functions theorem using the simplicity of the principal eigenvalue yields that the mapping $[0,(1-\delta)/2]\ni \xi\mapsto \lambda_1^\beta ((\kappa+1)\1_{(\xi,\xi+\delta)}-1)$ is differentiable, and in particular, so is the mapping $[0,(1-\delta)/2]\ni \xi\mapsto \lambda^\beta_*$. Let $\xi^*$ denote the optimal number $\xi$ minimizing $[0,(1-\delta)/2]\ni \xi\mapsto \lambda_1^\beta ((\kappa+1)\1_{(\xi,\xi+\delta)}-1)$. 

\paragraph{Step 3. Differentiation of the transcendental equation} Let us assume that $\xi^*\neq 0$. Then, we claim that $\left.\frac{\partial \lambda_1^\beta }{\partial\xi}\right|_{\xi=\xi^*}=0$. This claim follows immediately from the necessary first order optimality conditions if $\xi^*\in (0,(1-\delta)/2)$. If $\xi^*=(1-\delta)/2$, this is still true by using the symmetry property \eqref{sp0821} enjoyed by~$\lambda_1^\beta$. Therefore, assuming that $\xi^*\neq 0$, it follows that
$$
0=\left.\frac{\partial \lambda_1^\beta }{\partial\xi}\right|_{\xi=\xi^*}\frac{\partial F_{\alpha}}{\partial\lambda}(\xi^*,\beta,\lambda^\beta_*)+\frac{\partial F_{\alpha}}{\partial\xi}(\xi^*,\beta,\lambda^\beta_*)=\frac{\partial F_{\alpha}}{\partial\xi}(\xi^*,\beta,\lambda^\beta_*) 
$$
by differentiating \eqref{et0822} with respect to $\xi$. Let us compute $\frac{\partial F_{\alpha}}{\partial\xi}(\xi^*,\beta,\lambda^\beta_*)$. According to \eqref{Fet0825}, one has
\begin{multline*}
0=\frac{\partial F_{\alpha}}{\partial\xi}(\xi^*,\beta,\lambda^\beta_*) =  \frac{\partial F^s_{\alpha}}{\partial\xi}(\xi^*,\beta,\lambda^\beta_*)\sin(\sqrt{\lambda\kappa}\delta)\\
= -\sqrt{\lambda^\beta_*}(1+\kappa e^{2\alpha (1+\kappa)})(\lambda^\beta_*-\beta^2e^{2\alpha})\sinh\left(\sqrt{\lambda^\beta_*}(1-2\xi^*-\delta)\right)\sin\left(\sqrt{\lambda^\beta_* \kappa}\delta\right) .
\end{multline*}
Since $\sin\left(\sqrt{\lambda_*^\beta\kappa}\delta\right)> 0$ and $1-\delta-2\xi^*\geq 0$, one has either $\sinh\left(\sqrt{\lambda^\beta_*}(1-2\xi^*-\delta)\right)=0$, which yields $\xi^*=\frac{1-\delta}{2}$, or $\lambda^\beta_*=\beta^2e^{2\alpha}$.

The next result is devoted to the investigation of the equality 
\begin{equation}\label{eqlambdabeta}
\lambda^\beta_*=\beta^2e^{2\alpha}.
\end{equation}

\begin{lemma}\label{lemma:1547noel}
The equality \eqref{eqlambdabeta} holds true if, and only if $\beta=\beta_{\alpha,\delta}$, where $\beta_{\alpha,\delta}$ is defined by~\eqref{def:betaalphakappa}. Moreover, if $\beta <\beta_{\alpha,\delta}$, then  $\lambda^\beta_*>\beta^2e^{2\alpha}$ whereas if $\beta >\beta_{\alpha,\delta}$, then~$\lambda^\beta_*<\beta^2e^{2\alpha}$.
\end{lemma}
\begin{proof}
First notice that
\begin{equation}\label{train21251003}
\frac{\lambda_1^\beta (m_*^\beta)}{\beta^2} = \frac{1}{\beta}\min_{m\in \cM}\min_{\ph \in \cS} \left\{\frac {\int_\Omega e^{\alpha m} {|\nabla \ph|}^2} {\beta\int_\Omega m e^{\alpha m} \ph^2}+\frac {\int_{\partial \Omega} \ph^2} { \int_\Omega m e^{\alpha m} \ph^2}\right\}.
\end{equation}
where $\cS = \{\ph \in \mathrm{H}^1(\Omega),~ \int_\Omega m e^{\alpha m}\ph^2>0\}$ (see Eq. \eqref{def:lambda1beta}).
For $\varphi\in \cS$ and $m\in \cM$, the function $\R_+^* \ni\beta\mapsto \frac {\int_\Omega e^{\alpha m} {|\nabla \ph|}^2} {\beta\int_\Omega m e^{\alpha m} \ph^2}+\frac {\int_{\partial \Omega} \ph^2} { \int_\Omega m e^{\alpha m} \ph^2}$ is non-increasing, and therefore, so is the function $\R_+^* \ni\beta\mapsto \min_{m\in \cM}\min_{\ph \in \cS}  \frac {\int_\Omega e^{\alpha m} {|\nabla \ph|}^2} {\beta\int_\Omega m e^{\alpha m} \ph^2}+\frac {\int_{\partial \Omega} \ph^2} { \int_\Omega m e^{\alpha m} \ph^2}$. As a product of a non-increasing and a decreasing positive functions, we infer that the mapping~$\R_+\ni \beta \mapsto \lambda_1^\beta (m_*^\beta)/\beta^2 $ is decreasing, by using \eqref{train21251003}. Notice also that its range is $(0,+\infty)$.

Then, since Eq. \eqref{eqlambdabeta} also rewrites $\frac{\lambda_1^\beta (m_*^\beta)}{\beta^2} = e^{2\alpha}$, it has a unique solution $\beta_{\alpha,\delta}$ in~$\R_+^*$. Let us compute $\beta_{\alpha,\delta}$. One has
\begin{eqnarray*}
F^s_{\alpha}(\xi,\beta_{\alpha,\delta},\lambda_*^\beta)&=& \lambda_*^\beta (\kappa e^{2\alpha(\kappa+1)}-1) \left(\sinh (\sqrt{\lambda_*^\beta}(1-\delta))+\cosh (\sqrt{\lambda_*^\beta}(1-\delta))\right)\\
F^c_{\alpha}(\beta_{\alpha,\delta},\lambda_*^\beta) &=& 2\lambda_*^\beta \left(\sinh(\sqrt{\lambda_*^\beta}(1-\delta))+\cosh(\sqrt{\lambda_*^\beta}(1-\delta))\right)
\end{eqnarray*}
for every $\xi\in [0,(1-\delta)/2]$. It is notable that the previous quantities do not depend on $\xi$.

By plugging \eqref{eqlambdabeta} into the transcendental equation \eqref{et0822}, one gets that  $\beta_{\alpha,\delta}$ satisfies
$
(\kappa e^{2\alpha(\kappa+1)}-1) \sin \left(\sqrt{\kappa \lambda_*^\beta}\delta \right)=2\sqrt{\kappa}e^{\alpha (\kappa+1)}\cos \left(\sqrt{\kappa \lambda_*^\beta}\delta \right)
$
and in particular
$$
\tan \left(\sqrt{\kappa \lambda_*^\beta}\delta \right)=\frac{2\sqrt{\kappa}e^{\alpha (\kappa+1)}}{\kappa e^{2\alpha(\kappa+1)}-1},\quad \text{whenever }\kappa e^{2\alpha (\kappa+1)}\neq 1.
$$

The expected result hence follows easily  from the uniqueness of $\beta_{\alpha,\delta}$ and the fact that $\R_+\ni \beta \mapsto \lambda_1^\beta (m_*^\beta)/\beta^2 $ is decreasing.
\end{proof}

\paragraph{Step 4. Conclusion of the proof} We thus infer that, except if $\beta=\beta_{\alpha,\delta}$ one has the following alternative: either $\xi^*=0$ or $ \xi^*=\frac{1-\delta}{2}.$

In order to compute $\xi^*$, we will compare the real numbers $F_{\alpha}(0,\beta,\lambda)$ and $F_{\alpha}((1-\delta)/2,\beta,\lambda)$. Let us introduce the function $\Delta_{\alpha,\beta}$ defined by 
$$
\Delta_{\alpha,\beta}(\lambda) =\frac{F_{\alpha}(0,\beta,\lambda)-F_{\alpha}((1-\delta)/2,\beta,\lambda)}{\sin(\sqrt{\lambda \kappa }\delta)}.
$$
One computes
$
\Delta_{\alpha,\beta}(\lambda)= -\frac{1}{2}(\lambda-\beta^2e^{2\alpha})(\kappa e^{2\alpha(\kappa+1)}+1) \left(\cosh \left(\sqrt{\lambda}(1-\delta) \right)-1\right)
$ according to \eqref{Fet0825}.
According to Lemma \ref{lemma:1547noel}, one infers that
\begin{itemize}
\item if $\beta <\beta_{\alpha,\delta}$, then $\lambda_*^\beta >\beta^2e^{2\alpha}$ and $\Delta_{\alpha,\beta}(\lambda_*^\beta)<0$,
\item if $\beta >\beta_{\alpha,\delta}$, then $\lambda_*^\beta <\beta^2e^{2\alpha}$ and $\Delta_{\alpha,\beta}(\lambda_*^\beta)>0$.
\end{itemize}

Since $\lambda_*^\beta$ is the first positive zero of the transcendental equation for the parameter choice $\xi=\xi^*$, we need to know the sign of $F_{\alpha}(0,\beta,\lambda)$ and $F_{\alpha}((1-\delta)/2,\beta,\lambda)$ on the interval $[0,\lambda_*^\beta]$ to determine which function between $F_{\alpha}(0,\beta,\cdot)$ and $F_{\alpha}((1-\delta)/2,\beta,\cdot)$ vanishes at $\lambda_*^\beta$. For that purpose, we will compute the quantity $\partial F_{\alpha}(\xi,\beta,\lambda)/\partial \sqrt{\lambda}$ at~$\lambda=0$. According to \eqref{Fet0825}, one has
\begin{eqnarray*}
\left. \frac{\partial F_{\alpha}(\xi,\beta,\lambda)}{\partial \sqrt{\lambda}}\right|_{\lambda=0} &\!\!\!=& \sqrt{\kappa}\delta\left(- F_{\alpha}^s (\xi,\beta,\lambda) \cos (\sqrt{\lambda \kappa}\delta)+ e^{\alpha (\kappa+1)}\left.\frac{\partial F^c_{\alpha}(\xi,\beta,\lambda)}{\partial \sqrt{\lambda}}\cos (\sqrt{\lambda \kappa}\delta)\right)\right|_{\lambda=0}
\\
&\!\!\! = & \sqrt{\kappa}\delta \beta^2e^{2\alpha} +  \sqrt{\kappa}\delta e^{\alpha(\kappa+2)} \beta ( \beta e^{\alpha}(1-\delta) +2 )>0.
\end{eqnarray*}
As a result, since $F_{\alpha}(\xi,\beta,0)=0$ for every $\xi\in [0,(1-\delta)/2]$, the functions $F_{\alpha}(0,\beta,\cdot)$ and $F_{\alpha}((1-\delta)/2,\beta,\cdot)$ are both positive on $(0,\lambda_*^\beta)$ and according to the discussion on the sign of $\Delta_{\alpha,\beta}(\lambda_*^\beta)$ above, we infer that 
\begin{list}{--}{\itemsep0mm \topsep0mm}
\item if $\beta <\beta_{\alpha,\delta}$, then $\Delta_{\alpha,\beta}(\lambda_*^\beta)<0$ and $\xi^*=0$,  
\item if $\beta >\beta_{\alpha,\delta}$, then $\Delta_{\alpha,\beta}(\lambda_*^\beta)>0$ and $\xi^*=(1-\delta)/2$.
\end{list}

\section{Proof of Proposition \ref{prop:activeconstraint}}\label{sec:prooftheo:compm0}

In this section, we give the proof of Proposition \ref{prop:activeconstraint}. For this purpose, let $m_*^\beta$ be a solution of Problem \eqref{defSOP}, and assume by contradiction that $\int_0^1 m_*^\beta < -m_0$. Note that, as a consequence, the first order optimality conditions imply that $\eta^*=0$, and 
\begin{equation}\label{eq:rennes1052}
	\psi_0(x)\geq 0 ~\text{ for every }~ x\in \{m_*^\beta=-1\},
\end{equation}
where we use the same notations as those of Section \ref{subsec:bang-bang}.

We first assume that $\beta\leq \beta_{\alpha,\delta}$. 
As a consequence, Lemma \ref{lemma:1547noel} implies that $\lambda_*^\beta \geq \beta^2 e^{2\alpha}$. By Lemmas \ref{lem:metz1808} and \ref{lem:metz1816}, we know that $m_*^\beta$ is {\it bang-bang}, and a neighborhood of either $0$ or $1$ lies in $\{m_*^\beta=-1\}$. Assume that the former is true, and observe that since $\ph'(0)=\beta e^\alpha \ph(0)$, one has
$
	\psi_0(0) = (\alpha \beta^2e^{2\alpha} -\lambda_*^\beta(1-\alpha))\ph(0)^2.
$
The assumption $\alpha<1/2$ implies that $\alpha \beta^2e^{2\alpha} -\lambda_*^\beta(1-\alpha) \leq \lambda_*^\beta(2\alpha-1)<0$. As a consequence, \eqref{eq:rennes1052} yields that $\ph(0)=0$, and therefore $\ph'(0)=0$. Since $\ph$ satisfies $\ph''=\lambda_*^\beta \ph$ in a neighborhood of $0$,  we deduce that $\ph=0$ in this neighborhood, which is a contradiction. The case when $1\in \{m_*^\beta=-1\}$ is similar.

We now assume that $\alpha<\alpha_0$ and $\beta_{\alpha,\delta} < \beta < \infty$. 
Theorem~\ref{thm:main} implies that $m_*^\beta$ writes $m_*^\beta=(\kappa+1)\1_{(\xi,1-\xi)}-1$ for some $\xi\in (\xi^*,1/2]$, since $\int_0^1 m_*^\beta < -m_0$.
Observe that if $\ph$ is an eigenfunction of \eqref{EVP} associated to $\lambda_*^\beta$, then $\ph''=\lambda_*^\beta \ph$ on~$(0,\xi)$, and~$\ph'(0)=\beta e^\alpha \ph(0)$. As a consequence, for some constant $A> 0$ and for $x\in (0,\xi)$,
 one has
$\ph(x) = A\left(\sqrt{\lambda_*^\beta} \cosh \left(\sqrt{\lambda_*^\beta}x \right) 
		+ \beta e^\alpha \sinh \left(\sqrt{\lambda_*^\beta}x \right)\right)$. An easy computation shows that for every $x\in (0,\xi)$,
\begin{align}
	\psi_0(x) =\; & \textstyle
		\lambda_*^\beta A^2 \left( (2\alpha -1)\left((\lambda_*^\beta+\beta^2e^{2\alpha})
		\sinh^2 \left(\sqrt{\lambda_*^\beta}x \right) \right.\right. \notag\\
		& \left. \left.+2\beta \sqrt{\lambda_*^\beta} e^\alpha 
		\cosh \left(\sqrt{\lambda_*^\beta}x \right) \sinh \left(\sqrt{\lambda_*^\beta}x \right)
		\right)
		. +\; \alpha \beta^2e^{2\alpha} -(1-\alpha)\lambda_*^\beta \right). \label{eq:rennes2250}
\end{align}
We aim at proving that $\psi_0(\xi^*)<0$, which is in contradiction with \eqref{eq:rennes1052}. Noting that the terms $\lambda_*^\beta$, $2\beta \sqrt{\lambda_*^\beta} e^\alpha \cosh \left(\sqrt{\lambda_*^\beta}\xi^* \right) \sinh \left(\sqrt{\lambda_*^\beta}\xi^* \right)$ and $(1-\alpha)\lambda_*^\beta$ are all non-negative, we deduce from~\eqref{eq:rennes2250} that it is enough to prove that
\begin{equation}\label{eq:rennes1054}
	\textstyle
	\alpha < (1-2\alpha) \sinh^2 \left(\sqrt{\lambda_*^\beta}\xi^* \right).
\end{equation}
In the following, we note $\lambda_{1,\alpha}^\beta(m)$ (resp. $\lambda_{*,\alpha}^\beta$), instead of $\lambda_{1}^\beta(m)$ (resp. $\lambda_{*}^\beta$), in order to emphasize the dependency on $\alpha$. 
Since $\beta\mapsto \lambda_{1,\alpha}^\beta(m_*^\beta)$ is non-decreasing (see Appendix~\ref{sec:appendoptlocint}), one has
\begin{equation}\label{eq:rennes1145}
	\lambda_{*,\alpha}^\beta \;=\; \lambda_{1,\alpha}^\beta(m_*^\beta) \;\geq\; \lambda_{1,\alpha}^{\beta_{\alpha}^*} (m_*^\beta)
	\;\geq\; \lambda_{*,\alpha}^{\beta_{\alpha}^*}.
\end{equation}
Note that it is easily proved that the function $\alpha \mapsto \beta_{\alpha}^*$ is decreasing. 
Consequently, one has $\beta_{\alpha}^*>\beta_{1/2}^*$, and therefore $\lambda_{*,\alpha}^{\beta_{\alpha}^*} \geq \lambda_{*,\alpha}^{\beta_{1/2}^*}$. Moreover, Lemma~\ref{lemma:1547noel} yields that~$\lambda_{*,\alpha}^{\beta_{1/2}^*} > {\big(\beta_{1/2}^*\big)}^2 e^{2\alpha} \geq  {\big(\beta_{1/2}^*\big)}^2$. Combining this last inequality with \eqref{eq:rennes1145}, we obtain that $\sqrt{\lambda_{*,\alpha}^\beta} \geq \beta_{1/2}^*$. As a consequence, since $\xi^*< \xi$, one has
$
	\alpha < \alpha_0 < \frac{\sinh^2{\left(\sqrt{\lambda_{*,\alpha}^\beta} \xi \right)}}{1+2\sinh^2{\left(\sqrt{\lambda_{*,\alpha}^\beta} \xi \right)}},
$
which yields \eqref{eq:rennes1054} and achieves the proof of the result when $\beta>\beta_{\alpha}^*$.

We are left with dealing with the case $\beta=\infty$. Observe that in this case the function $\psi_0$ takes the simpler form
$ \textstyle
	\psi_0(x) =\lambda_*^\infty A^2 \left( (2\alpha -1) \sinh^2 \left(\sqrt{\lambda_*^\infty}x \right) 
	+\alpha\right).
$
As a consequence, the assumption $\alpha<\alpha_0$ still implies that $\psi_0(\xi^*)<0$, and the previous reasoning holds, which concludes the proof.


\bibliographystyle{abbrv}
\bibliography{biblio}

\begin{thebibliography}{10}

\bibitem{MR1469392}
G.~A. Afrouzi and K.~J. Brown.
\newblock On principal eigenvalues for boundary value problems with indefinite
  weight and {R}obin boundary conditions.
\newblock {\em Proc. Amer. Math. Soc.}, 127(1):125--130, 1999.

\bibitem{MR1372792}
F.~Belgacem and C.~Cosner.
\newblock The effects of dispersal along environmental gradients on the
  dynamics of populations in heterogeneous environments.
\newblock {\em Canad. Appl. Math. Quart.}, 3(4):379--397, 1995.

\bibitem{BLR}
H.~Berestycki and T.~Lachand-Robert.
\newblock Some properties of monotone rearrangement with applications to
  elliptic equations in cylinders.
\newblock {\em Math. Nachr.}, 266:3--19, 2004.

\bibitem{Brezis}
H.~Brezis.
\newblock {\em Functional analysis, {S}obolev spaces and partial differential
  equations}.
\newblock Universitext. Springer, New York, 2011.

\bibitem{MR1112065}
R.~Cantrell and C.~Cosner.
\newblock Diffusive logistic equations with indefinite weights: population
  models in disrupted environments. {II}.
\newblock {\em SIAM J. Math. Anal.}, 22(4):1043--1064, 1991.

\bibitem{MR1105497}
R.~S. Cantrell and C.~Cosner.
\newblock The effects of spatial heterogeneity in population dynamics.
\newblock {\em J. Math. Biol.}, 29(4):315--338, 1991.

\bibitem{MR1961241}
C.~Cosner and Y.~Lou.
\newblock Does movement toward better environments always benefit a population?
\newblock {\em J. Math. Anal. Appl.}, 277(2):489--503, 2003.

\bibitem{MR1423004}
S.~Cox and R.~Lipton.
\newblock Extremal eigenvalue problems for two-phase conductors.
\newblock {\em Arch. Rational Mech. Anal.}, 136(2):101--117, 1996.

\bibitem{MR2660987}
A.~Derlet, J.-P. Gossez, and P.~Tak{\'a}{\v{c}}.
\newblock Minimization of eigenvalues for a quasilinear elliptic {N}eumann
  problem with indefinite weight.
\newblock {\em J. Math. Anal. Appl.}, 371(1):69--79, 2010.

\bibitem{MR2205529}
Y.~Du.
\newblock {\em Order structure and topological methods in nonlinear partial
  differential equations. {V}ol. 1}, volume~2 of {\em Series in Partial
  Differential Equations and Applications}.
\newblock World Scientific Publishing Co. Pte. Ltd., Hackensack, NJ, 2006.
\newblock Maximum principles and applications.

\bibitem{HenrotBook}
A.~Henrot.
\newblock {\em Extremum problems for eigenvalues of elliptic operators}.
\newblock Frontiers in Mathematics. Birkh\"auser Verlag, Basel, 2006.

\bibitem{henrot-pierre}
A.~Henrot and M.~Pierre.
\newblock {\em Variation et optimisation de formes}, volume~48.
\newblock Springer-Verlag Berlin Heidelberg, 2005.

\bibitem{MR1100011}
P.~Hess.
\newblock {\em Periodic-parabolic boundary value problems and positivity},
  volume 247 of {\em Pitman Research Notes in Mathematics Series}.
\newblock Longman Scientific \& Technical, Harlow; copublished in the United
  States with John Wiley \& Sons, Inc., New York, 1991.

\bibitem{MR588690}
P.~Hess and T.~Kato.
\newblock On some linear and nonlinear eigenvalue problems with an indefinite
  weight function.
\newblock {\em Comm. Partial Differential Equations}, 5(10):999--1030, 1980.

\bibitem{HinKaoLau12}
M.~Hinterm\"uller, C.-Y. Kao, and A.~Laurain.
\newblock Principal eigenvalue minimization for an elliptic problem with
  indefinite weight and {R}obin boundary conditions.
\newblock {\em Appl. Math. Optim.}, 65(1):111--146, 2012.

\bibitem{JhaPor11}
K.~Jha and G.~Porru.
\newblock Minimization of the principal eigenvalue under {N}eumann boundary
  conditions.
\newblock {\em Numer. Funct. Anal. Optim.}, 32(11):1146--1165, 2011.

\bibitem{KaoLouYanagida}
C.-Y. Kao, Y.~Lou, and E.~Yanagida.
\newblock Principal eigenvalue for an elliptic problem with indefinite weight
  on cylindrical domains.
\newblock {\em Math. Biosci. Eng.}, 5(2):315--335, 2008.

\bibitem{MR810619}
B.~Kawohl.
\newblock {\em Rearrangements and convexity of level sets in {PDE}}, volume
  1150 of {\em Lecture Notes in Mathematics}.
\newblock Springer-Verlag, Berlin, 1985.

\bibitem{LLNP2016}
J.~Lamboley, A.~Laurain, G.~Nadin, and Y.~Privat.
\newblock Properties of optimizers of the principal eigenvalue with indefinite
  weight and {R}obin conditions.
\newblock {\em Calc. Var. Partial Differential Equations}, 55(6):55:144, 2016.

\bibitem{Lou}
Y.~Lou.
\newblock Some challenging mathematical problems in evolution of dispersal and
  population dynamics.
\newblock In {\em Tutorials in mathematical biosciences. {IV}}, volume 1922 of
  {\em Lecture Notes in Math.}, pages 171--205. Springer, Berlin, 2008.

\bibitem{MR2281509}
Y.~Lou and E.~Yanagida.
\newblock Minimization of the principal eigenvalue for an elliptic boundary
  value problem with indefinite weight, and applications to population
  dynamics.
\newblock {\em Japan J. Indust. Appl. Math.}, 23(3):275--292, 2006.

\bibitem{MR2455723}
J.-M. Rakotoson.
\newblock {\em R\'earrangement relatif}, volume~64 of {\em Math\'ematiques \&
  Applications (Berlin) [Mathematics \& Applications]}.
\newblock Springer, Berlin, 2008.
\newblock Un instrument d'estimations dans les probl{\`e}mes aux limites. [An
  estimation tool for limit problems].

\bibitem{Roques-Hamel}
L.~Roques and F.~Hamel.
\newblock Mathematical analysis of the optimal habitat configurations for
  species persistence.
\newblock {\em Math. Biosci.}, 210(1):34--59, 2007.

\bibitem{MR0043440}
J.~G. Skellam.
\newblock Random dispersal in theoretical populations.
\newblock {\em Biometrika}, 38:196--218, 1951.

\end{thebibliography}

\end{document}